\documentclass[10pt,a4paper]{amsart}
\usepackage[latin1]{inputenc}
\usepackage{amsmath}
\usepackage{amsfonts}
\usepackage{amssymb}
\usepackage{graphicx}
\numberwithin{equation}{section}
\usepackage{caption}
\usepackage[colorlinks=true]{hyperref}
\graphicspath{{figsFPL1/}}

\author{Philippe Nadeau}

\newtheorem{thm}{Theorem}[section]
\newtheorem{prop}[thm]{Proposition}
\newtheorem{lemma}[thm]{Lemma}
\newtheorem{cor}[thm]{Corollary}
\newtheorem{conj}[thm]{Conjecture}

\theoremstyle{definition}
\newtheorem{defi}[thm]{Definition}
\newtheorem{rem}[thm]{Remark}

\newcommand{\dem}{\noindent \textbf{Proof: }}
\newcommand{\demof}[1]{\noindent \textbf{Proof of~{#1}:}}
\newcommand{\findem}{\vspace{-.55cm} \begin{flushright} $\square~$
\end{flushright} \vspace{.2cm} }

\newcommand{\si}{\sigma}
\newcommand{\eps}{\varepsilon}
\newcommand{\tspt}{{t_{\si,\tau}^\pi}}
\newcommand{\Tspt}{{\mathcal{T}_{\si,\tau}^\pi}}
\newcommand{\oTspt}{{\overrightarrow{\mathcal{T}}_{\si,\tau}^\pi}}

\newcommand{\ssyt}{\operatorname{SSYT}}

\newcommand{\Dn}{\mathcal{D}_n}
\newcommand{\zero}{\mathbf{0}_n}
\newcommand{\un}{\mathbf{1}_n}
\renewcommand{\b}{\mathbf{b}}
\newcommand{\wb}{\widetilde{\mathbf{b}}}
\renewcommand{\t}{\mathbf{t}}
\newcommand{\ttau}{\mathbf{t}_{*\tau}}
\renewcommand{\c}{\mathbf{c}}
\newcommand*{\xpimk}{X(\pi,m,k)}
\renewcommand{\deg}{d}

\title{Fully Packed Loop configurations in a triangle}

\begin{document}

\begin{abstract}
 Fully Packed Loop configurations (FPLs) are certain configurations on the square grid, naturally refined according to certain link patterns. If $A_X$ is the number of FPLs with link pattern $X$, the Razumov--Stroganov correspondence provides relations between numbers $A_X$ relative to a given grid size. In another line of research, if $X\cup p$ denotes $X$ with $p$ additional nested arches, then $A_{X\cup p}$ was shown to be polynomial in $p$: the proof gives rise to certain configurations of FPLs in a triangle (TFPLs).

 In this work we investigate these TFPL configurations and their relation to FPLs. We prove certain properties of TFPLs, and enumerate them under special boundary conditions. From this study we deduce a class of linear relations, conjectured by Thapper, between quantities $A_X$ relative to different grid sizes, relations which thus differ from the Razumov--Stroganov ones. 
\end{abstract}

\maketitle 

\section*{Introduction}

Fully Packed Loop configurations, or FPLs, are certain subgraphs of a square grid $n\times n$, in simple bijection with alternating sign matrices and other combinatorial structures~\cite{propp}. The total number $A_n$ of such FPLs is thus known since the works of Kuperberg~\cite{Kup-ASM} and Zeilberger~\cite{Zeil-ASM}.

One specificity of FPLs over the other objects in bijection is that there exists a natural way to partition them according to certain link patterns (see Section~\ref{sect:fpl_link}); this partition is far from obvious on any other representation. If $n$ is an integer and $\pi$ is a link pattern with $n$ arches, we let $A_\pi$ be the number of FPLs which induce $\pi$; then $A_n$ is the sum of $A_\pi$ for all link patterns  $\pi$ with $n$ arches. 

The interest in the numbers $A_\pi$ was initially motivated by the Razumov--Stroganov correspondence, conjectured in~\cite{RS-conj} and proved in~\cite{ProofRS}. This correspondence consists of certain linear relations between the numbers $A_\pi$, which essentially characterize these numbers. Nevertheless some ``nice'' expressions for the numbers $A_\pi$ are still not known, and in this work we explore one possible path to such expressions.

More precisely, the present work is directly inspired by the two articles~\cite{CKLN} and~\cite{Thapper}; in the former article a certain combinatorial decomposition of FPLs based on \textit{FPLs in a triangle} (TFPLs) was investigated, while the latter gave numerous conjectures and partial results regarding these TFPLs. These TFPLs are indexed by three $01$-words $\si,\pi,\tau$, and their number is noted $\tspt$.
\medskip

The main contributions of the present article are the following:

\begin{itemize}
\item A much simplified proof of the key technical result ~\cite[Theorem 7.1]{CKLN}. This is accomplished thanks to the introduction of \emph{oriented} TFPLs; this is Lemma~\ref{lem:Ni_f} and the remarks following it.
\item A proof of a Conjecture 3.4 in~\cite{Thapper}, which is Theorem~\ref{th:conjthapper} in the present article. This leads to certain linear relations among the numbers $A_\pi$ given in Theorem \ref{th:linearcomb}
\item A study of the numbers $\tspt$ when $\si$ and $\pi$ have common prefixes/suffixes (Theorems~\ref{th:commonps} and~\ref{th:rotatedpart}).
\end{itemize}

The article is organized as follows: in Section~\ref{sect:fpl_link} we give a precise definition of Fully Packed Loop configurations. Section~\ref{sect:squaretotriangle} reviews the key ideas from~\cite{CKLN,Thapper}, and thus explains how, from the enumeration of TFPL configurations, one can recover the enumeration of usual FPL configurations. We will follow the aforementioned articles --improving the arguments in some places-- and fix notations. Section~\ref{sect:tspt} presents several properties of the TFPL numbers $\tspt$; in particular, we introduce the concept of \emph{oriented} TFPL configurations, which we use to give a concise proof of
Theorem~\ref{th:siinfpi}, which was first proved in~\cite{CKLN} after a lengthy
case by case analysis. Enumeration of TFPLs in a special case is dealt with in 
Theorems~\ref{th:commonps} and~\ref{th:rotatedpart}; their proof  in
Section~\ref{sect:proofcommonps} also uses oriented TFPLs.
Section~\ref{sect:linear} follows an idea of Thapper (motivated by certain
conjectures of Zuber~\cite{Zuber-conj}) which uses TFPLs to determine certain
linear relations between refined FPL numbers. There we prove a conjecture of
Thapper related to TFPL numbers, and give some properties and conjectures about
the coefficients $c_{\alpha\pi}$ involved in the linear relations.
\medskip

In a forthcoming paper~\cite{NadFPL2}, we continue this program and show that a certain subclass of TFPLs turns out to be enumerated by Littlewood Richardson coefficients. Let us also mention the paper~\cite{ZJtriangle}, where the author conjectures a certain expression for TFPLs, which in particular would give a new proof for the Razumov--Stroganov correspondence; this expression is in any case coherent with some of ours results (cf~\cite[note on page 20]{ZJtriangle}).
\medskip

\emph{Note: This article contains some of the work presented at the Fpsac 2010 conference in San Francisco~\cite{NadFPLFpsac}. }

\section{Fully Packed Loop configurations and Link patterns}
\label{sect:fpl_link}

\subsection{Fully Packed Loop configurations}
\label{sub:fpl}
We fix a positive integer $n$, and let $G_n$ be the square grid with $n^2$ vertices, together with $4n$ {\em external edges}: see Figure~\ref{fig:grid} for an illustration in the case $n=7$. Note that we consider such external edges as half edges actually, which means that they are attached to one vertex only instead of two. For convenience we number these external edges from $1$ to $4n$ in counterclockwise order, starting from the leftmost external edge on the bottom boundary. We remark also that each vertex is incident to precisely $4$ edges.

\begin{figure}[!ht]
\includegraphics[width=0.3\textwidth]{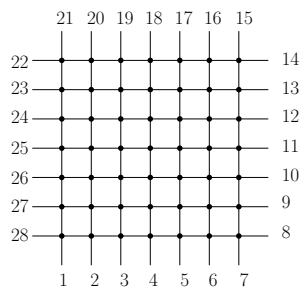}
\caption{The grid $G_7$
\label{fig:grid}}
\end{figure}

\begin{defi}[Fully Packed Loop Configurations]
A Fully Packed Loop configuration (FPL) $F$ of size $n$ is a subgraph of $G_n$  such that:
\begin{enumerate}
\item each vertex of $G_n$ is incident to two edges of $F$;
\item The external edges belonging to $F$ are either the odd labeled external edges or the even labeled ones.
\end{enumerate}
\end{defi}

An example of configuration is given on Figure~\ref{fig:fplexample}. Note that there is a simple bijection between FPL configurations with even and odd labeled external edges, by reflecting against a diagonal of the grid. It is handy to leave the definition as such and not privilege one boundary condition over the other though, in particular for Wieland's rotation (Theorem~\ref{th:wieland}). For enumeration purposes, we let $FPL(n)$ be the number of FPL configurations of size $n$ with odd-labeled external edges; from the previous remarks, $FPL(n)$ also enumerates configurations with even-labeled external edges.

\begin{figure}[!ht]
\includegraphics[width=0.4\textwidth]{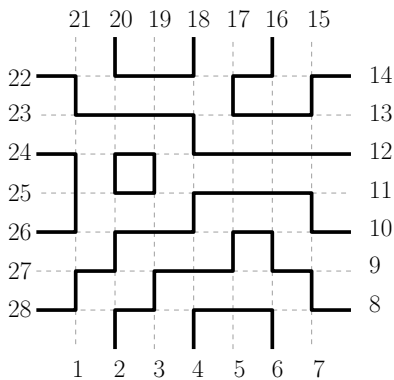}
\caption{A FPL configuration of size $7$.
\label{fig:fplexample}}
\end{figure}

FPL configurations are known to be in bijection with numerous combinatorial  objects, the most important being {\em alternating sign matrices}, whose enumeration was conjectured in~\cite{MRR-Mac}, proved first by Zeilberger, and shortly after a shorter proof was given by Kuperberg:

\begin{thm}[\cite{Kup-ASM,Zeil-ASM}]
For any positive integer $n$ we have

\begin{equation*}
FPL(n)=\prod_{i=0}^{n-1} \frac{(3i+1)!}{(n+i)!}
\end{equation*}
\end{thm}

In this work we will be interested by a refinement of FPLs, which is given by attaching to each configuration a {\em link pattern} which we now define.


\subsection{Link Patterns and Wieland's rotation}
We define a \emph{link pattern} $X$ of size $n$ to be a noncrossing matching on $I$, where $I$ is any subset of the integers of cardinality $2n$. In other words, it is a partition of $I$ into pairs such that there are no integers $a<b<c<d$ such that $\{a,c\}$ and $\{b,d\}$ are both in $X$.

 Now, since all vertices are of degree $2$ in an FPL configuration $F$, edges in $F$ are arranged in paths, which can be either closed or not, and in this last case both extremities of the paths are external edges. Therefore one can associate to $F$ a link pattern:

\begin{defi}[Link pattern $X(F)$]
Given an FPL configuration, the link pattern $X(F)$ is the set of pairs $\{i,j\}$ of labels of external edges, where $\{i,j\}$ belongs to $X(F)$ if and only if $i$ and $j$ are the labels of the extremities of the same path in $F$.
\end{defi}

For instance, if $F$ is the configuration of Figure~\ref{fig:fplexample}, then we have 
\[X(F)=\{\{2,8\},\{4,6\},\{10,28\},\{12,22\},\{14,16\},\{18,20\},\{24,26\}\}.\]

\begin{defi}[$\mathcal{A}_X$ and $A_X$]
Let $X$ be an odd or even link pattern. The set $\mathcal{A}_X$ is defined as the set of all $FPL$ configurations $F$ of size $n$ such that $X(F)=X$. We define also $A_X$ as the cardinal $|\mathcal{A}_X|$. 
\end{defi}

We fix now a link pattern $X$ as in the previous definition and consider the {\em rotated} link pattern $r(X)$ defined by:
 \[
\{i,j\}\in r(X)\text{  if and only if  }\{i-1,j-1\}\in X,
\]
 where indices are taken modulo $4n$. Note that $r$ corresponds geometrically to a counterclockwise rotation, and that it sends odd-labeled external edges to even-labeled ones and conversely. We have then the beautiful result of Wieland:

\begin{thm}[\cite{Wieland}]
\label{th:wieland}
For any link pattern $X$, we have the equality $A_X=A_{r(X)}$.
\end{thm}

\dem We will not give the complete proof, but we define the relevant bijection $W$ from $\mathcal{A}_X$ to $\mathcal{A}_{r(X)}$  from~\cite{Wieland} since we use it in the proof of Theorem~\ref{th:conjthapper}. Let the {\em cells} of $G_n$ be the $(n+1)^2$ unit squares together with their surrounding edges, including external cells that have $2$ or $3$ surrounding edges only. We partition cells in a chessboard manner to get even and odd cells, where by convention the cells lying on the Southwest-Northeast diagonal are even. Define the \emph{active cells} to be the even (respectively odd) ones if $X$ is a link pattern between even (resp. odd) edges.

\begin{figure}[!ht]
\includegraphics{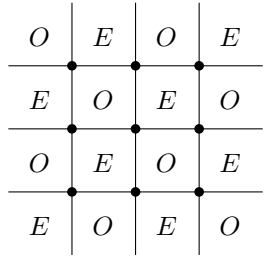}
\caption{Odd (O) and Even (E) squares.
\label{fig:oddeven}}
\end{figure}

Now let $F$ be a configuration in $\mathcal{A}_X$, and let $c$ be any of its cells. We distinguish two cases: if there are precisely two edges of $F$ on opposite sides of $c$ (either horizontal or vertical), we define $U(c)=c$. Otherwise, we define $U(c)$
as the same cell where we exchange edges of $F$ and those that do not belong to $F$. Now, given $F$, apply $U$ to all {\em active} cells of $F$, and let the result be $W(F)$; note that the order in which $U$ is applied on active cells is irrelevant since any two of these cells share no edge. An illustration of the bijection is provided on Figure~\ref{fig:wieland}.

\begin{figure}[!ht]
\includegraphics{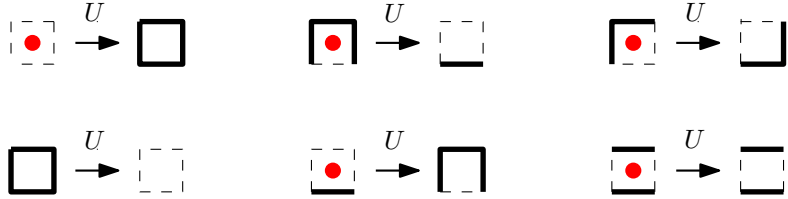}
\caption{The action of $U$ on active cells.
\label{fig:wieland_elem}}
\end{figure}

 It is shown in~\cite{Wieland} that $W$ is a bijection from $\mathcal{A}_X$ to $\mathcal{A}_{r(X)}$, and the reciprocal construction $W^{-1}$ consists in applying $U$ to cells which are {\em not} active.
\findem

This result has a nice consequence for enumeration, since one can rotate a link pattern to get a more pleasant, but equivalent counting problem: this technique is applied in~\cite{CasKrat,CKLN,DZ4}. Furthermore, we will make use of it in Section~\ref{sub:xpimk} as already mentioned, and it is at the heart of the proof of the Razumov--Stroganov conjecture in~\cite{ProofRS}.

\begin{figure}[!ht]
\includegraphics[width=0.7\textwidth]{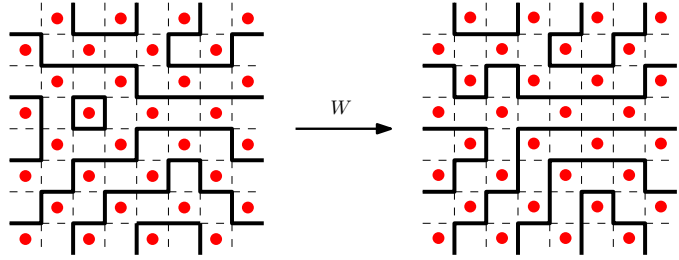}
\caption{Example of Wieland rotation. Marked cells are active on the left side.
\label{fig:wieland}}
\end{figure}


\subsection{\texorpdfstring{Definitions}{Definitions}}
\label{sub:Dn}

We consider finite words on the alphabet $\{0,1\}$, which we call $01$-words. For $u$ a $01$-word, we let $|u|_0$ denote the number of $0$s it contains, $|u|_1$ the number of $1$s it contains, and $|u|=|u|_0+|u|_1$ its total number of letters.

\begin{prop}
\label{prop:bijwordsdiagrams}
Given nonnegative integers $k,\ell$, there is a bijection between:\\
$(a)$ words $\si$ such that $|\si|_0=k$ and  $|\si|_1=\ell$, and\\
$(b)$ Ferrers diagrams fitting in the rectangle with $k$ rows and $\ell$ columns.
\end{prop}
\begin{proof}
 This is very standard. Given such a word $\si=\si_1\cdots \si_{k+\ell}$, construct a path on the square lattice by drawing a North step when $\si_i=0$ and an East step when $\si_i=1$, for $i$ from $1$ to $k+\ell$. Then complete the picture by drawing a line up from the starting point, and a line to the left of the ending point. The resulting region enclosed in the wanted Ferrers diagram; see Figure~\ref{fig:bijwordsdiagrams} for an example.
\end{proof}

\begin{figure}[!ht]
\begin{center}
\includegraphics[width=0.4\textwidth]{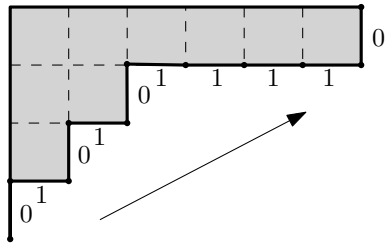}
\caption{Bijection between words and Ferrers diagrams.
\label{fig:bijwordsdiagrams}}
\end{center}
\end{figure}

We need to define several notions for words, most of them coming from the Ferrers diagram representation.
\medskip

We define the \emph{conjugate} $\si^*$ of $\si=\si_1\cdots\si_{n}$ by $\si^*_i:=1-\si_{n+1-i}$; in the example of Figure \ref{fig:bijwordsdiagrams} we get $\si^*=1000010101$. This corresponds to reflecting a Ferrers diagram with respect to its main diagonal. The {\em inversion number} of $\si$ is the number of indices $i<j$ such that $(\si_i,\si_j)=(1,0)$, and is noted $\deg(\si)$; this is the number of boxes in the Ferrers diagram representation. For instance we have $\deg(\si)=9$ for the example of Figure~\ref{fig:bijwordsdiagrams}.

Now suppose $\si,\tau$ verify $|\si|_0=|\tau|_0$ and $|\si|_1=|\tau|_1$, so that they form Ferrers diagrams included in a common rectangle.
 We define a partial order by $\si\leq \tau$ if $\si$ is included in $\tau$ in the diagram representation. Define the partial sum $\si_{\leq i}=\sum_{j\leq i}\si_j$ (which is just the number of $1$s with index less or equal to $j$). Then it is easily seen that $\si\leq \tau$ if and only if $\si_{\leq i}\leq \tau_{\leq i}$ for all indices $i$. 

We now define more notions related to diagrams: if $\si\leq \tau$, we define $\tau/\si$ as the set of boxes that are in $\tau$ but not in $\si$; if there are no two boxes in the same column (respectively row) of $\tau/\si$, then $\tau/\si$ is a {\em horizontal strip} (resp. {\em vertical strip}), and we write $\si\to\tau$. The size $\deg(\tau/\si)$ of a skew shape is naturally $\deg(\tau)-\deg(\si)$.

 We define a {\em semistandard Young tableau} of shape $\si$ of length $N\geq 0$ to be a sequence $(\si^{(i)})_{i=0\ldots N}$ of words, where $\si^{(0)}=\zero\leq\si^{(1)}\ldots\leq\si^{(N)}=\si$ and $\si^{(i+1)}/\si^{(i)}$ is a horizontal strip for all $i<N$. This is equivalent to the following usual definition: a filling of the boxes of the diagram $\si$ by positive integers not bigger than $N$, which are nondecreasing in each row from left to right, and increasing in each column from top to bottom.

 Suppose that $u$ is a box in the diagram $\si$, which is in the $k$th row from the top and $\deg$th column from the left. The {\em content} $c(u)$ of $u$ is defined as $\deg-k$, while its {\em hook-length} $h(u)$ is defined as the number of boxes in $\si$ which are below $u$ and in the same column, or right of $u$ and in the same row ($u$ itself being counted just once). We have then:

\begin{thm}[The hook content formula~\cite{StanPP}]
\label{th:hookcontent}
The number of semistandard Young tableaux of shape $\si$ and length $N\geq 0$ is given by 
 \begin{equation}
\prod_{u\in\si}\frac{N+c(u)}{h(u)}
\end{equation}
\end{thm}

We define $SSYT(\si,N)$ as this quantity considered as a {\em polynomial} in $N$. It has leading term $\frac{1}{H(\si)}N^{\deg(\si)}$ where $H(\si):=\prod_{u\in\si}{h(u)}$.

%
\medskip

We will be particularly interested in the following set of words:

\begin{defi}[$\Dn$]
We denote by $\mathcal{D}_n$ the set of words $\si$ on the alphabet $\{0,1\}$ of length $2n$, such that $|\si|_0=|\si|_1$ and each prefix $u$ of $\si$ verifies $|u|_0\geq|u|_1$.
\end{defi}

This is simply the set of well parenthesized words of length $2n$, known as \emph{Dyck words}, which are counted by the Catalan number $|\Dn|=C_n:=\frac{1}{n+1}\binom{2n}{n}$. 

\begin{prop}
\label{prop:Dn_objects}
Given a nonnegative integer $n$, there are explicit bijections between:
\begin{enumerate}
\item The set $\Dn$;
\item Ferrers diagrams included in the staircase diagram $\delta_n:=(n-1,n-2,\ldots,1)$;
\item \label{pouet} non crossing matchings on the set $\{1,2,\ldots,2n\}$.
\end{enumerate}
\end{prop}

\dem This is all very standard; we explicit here the bijections that we will use to identify the 3 objects. These identifications are illustrated on Figure~\ref{fig:Dn}. To go from $\Dn$ to Ferrers diagrams, this is just the bijection of Proposition~\ref{prop:bijwordsdiagrams} (for $k=\ell=n$) restricted to $\Dn$. Now given a matching as in $(3)$, define an element $\si$ of $\Dn$ in the following way: for any pair $\{i,j\}$ in the matching with $i<j$, set $\si_i:=0$ and $\si_j:=1$. 
\findem

\begin{figure}[!ht]
\includegraphics[width=0.6\textwidth]{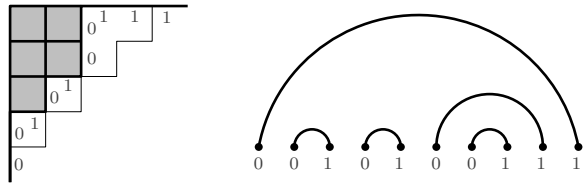}
\caption{The word $0010100111\in \mathcal{D}_5$ under different representations.
\label{fig:Dn}}
\end{figure}

Note that $(\Dn,\leq)$ is a poset, with smallest element $\zero:=0^n1^n$ and greatest element $\un:=(01)^n$; they are respectively the empty diagram and the diagram $\delta_n$, and the poset is isomorphic to an interval in Young's lattice.

\section{From the square to the triangle}
\label{sect:squaretotriangle}

In this section we will recall the general setting of the articles~\cite{CKLN,Thapper}, thereby fixing notations and simplifying some arguments.

\subsection{\texorpdfstring{The link patterns $\xpimk$: nested arches}{Nested link patterns}}
\label{sub:xpimk}

Suppose we have a non crossing matching $\pi$ on $\{1,\ldots,2n\}$, and  a nonnegative integer $m$. We first consider the matching
\[
\pi_{(m)}:=0^m\pi 1^m;                                                                                \]
 that is, we add $m$ nested arches around the matching $\pi$. Now we fix an extra integer $k\geq 0$, and consider the link pattern $\xpimk$ on the grid $G_{n+m}$ defined as follows: $k+2i-1$ and $k+2j-1$ are matched in  $\xpimk$ if and only if $i$ and $j$ are matched in $\pi_{(m)}$, where indices on the grid are taken modulo $4(n+m)$. An example of $\xpimk$ is given on Figure~\ref{fig:xpimk}.
\begin{figure}[!ht]
\includegraphics[width=0.5\textwidth]{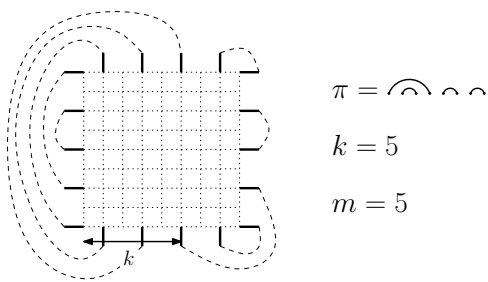}
\caption{A link pattern of the form $\xpimk$.
\label{fig:xpimk}}
\end{figure}
%

Now an immediate consequence of Wieland's rotation (Theorem~\ref{th:wieland}) is that:
for any $\pi,m,k$ as above, $A_{X(\pi,m,k+1)}=A_{\xpimk}$. So this value is independent of $k$, and we can thus define the following:
\begin{defi}
 \label{def:apim}
For $\pi,m,k$ as before, we define $A_\pi (m)$ to be the number $A_{\xpimk}$ for any value of $k$. 
\end{defi}

Note that in particular we have $A_\pi(0)=A_X$ where $X=X(\pi,0,0)$ can be any link pattern, so that studying the $A_\pi(m)$ is equivalent to studying the $A_X$.
\medskip

 \emph{We assume from now on and until Theorem~\ref{th:mainckln} that we have the inequality $m\geq 3n+k-1$}: this condition ensures that all external edges corresponding to $\pi$ appear on the bottom boundary of the grid $G_{n+m}$. 

This condition is verified on Figure~\ref{fig:thebigpicture}. Note that numerous edges inside the grid are already drawn on the picture. The reason is that all these edges are part of all FPL configurations with link pattern $\xpimk$: they are said to be {\em fixed} with respect to the link pattern. The basic lemma to prove that these edges are fixed is due to de Gier~\cite[Lemma 39]{degierloops}; how one applies  this lemma  in our case is done in detail in~\cite{CKLN} and~\cite{Thapper}, so we will not repeat this here.

\begin{figure}[!ht]
\includegraphics[width=0.8\textwidth]{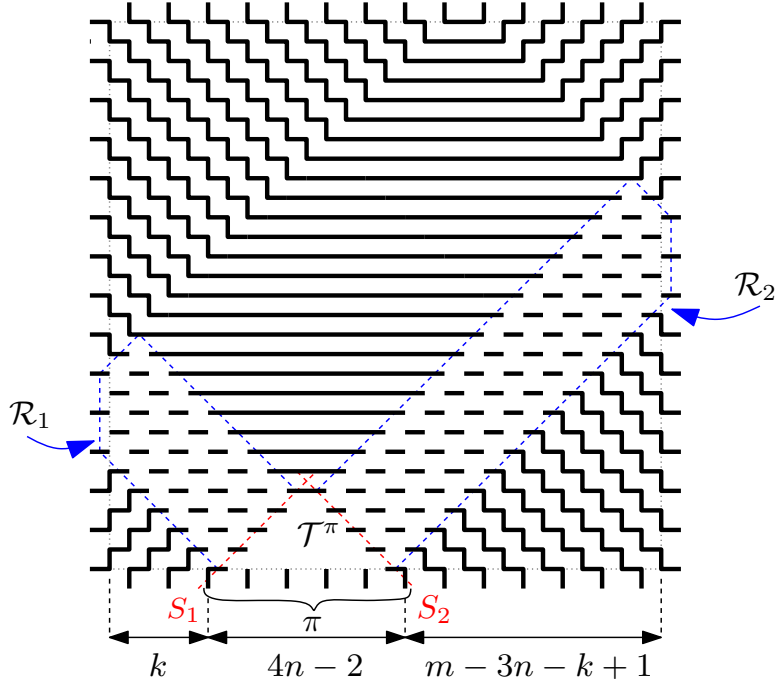}
\caption{Fixed edges for $\xpimk$.
\label{fig:thebigpicture}}
\end{figure}

These fixed edges give rise to various regions in the grid: two of them are pentagons which we note $\mathcal{R}_1$ and $\mathcal{R}_2$, which are both adjacent to a third region $\mathcal{T}$ with a triangular shape. For this last region, we will sometimes write it as $\mathcal{T}^n$ to stress the dependence in $n$, and as  $\mathcal{T}^\pi$ to indicate that a particular matching $\pi$ has been chosen on it bottom boundary; note also that   $\mathcal{R}_1$ depends on $k$ and $n$, while  $\mathcal{R}_2$ depends on $k,m$ and $n$. The rest of the grid is completely determined outside of these three regions, since all vertices there already have two incident fixed edges. The border between the region $\mathcal{R}_1$ and $\mathcal{T}$ (respectively $\mathcal{R}_2$  and $\mathcal{T}$) is indicated by a segment denoted $S_1$ (resp. $S_2$). These two segments cross precisely $2n$ vertical edges of the grid $G_{n+m}$, which we will write $v_1,\ldots,v_{2n}$ from bottom to top for $S_1$, and $w_1,\ldots,w_{2n}$ from top to bottom for $S_2$.

\begin{defi}
\label{defi:sigmatau}
Let $F$ be a FPL configuration verifying the link pattern $\xpimk$. Then we encode the border between $\mathcal{R}_1$ and $\mathcal{T}$ by a sequence $\si(F)=\si_1\cdots \si_{2n}$ where $\si_i=0$ if $v_i$ belongs to $F$, and $v_i=1$ otherwise.

 We encode the border between $\mathcal{R}_2$ and $\mathcal{T}$ by a sequence $\tau(F)=\tau_1\cdots \tau_{2n}$ where $\tau_i=0$ if $w_i$ does not belong to $F$ and $w_i=1$ otherwise.
\end{defi}

Note that $\si$ and $\tau$ have asymmetric interpretations, and that in~\cite{Thapper, ZJtriangle}, other conventions are taken. Now we have the following crucial result:

\begin{prop}
\label{prop:siDn}
For any configuration $F$ verifying $\xpimk$, the sequences $\si(F)$ and $\tau(F)$ belong to $\Dn$.
\end{prop}

\dem
By symmetry it is enough to do it for $\si(F)$. Now one first shows that $\si(F)$ has $n$ zeros and $n$ ones, and this is done by a simple counting of how many paths in $F$ have to go through the triangle (see~\cite[p. 14]{CKLN}). To show that it actually has no more ones than zeros in each prefix is a consequence of the proof of Theorem~\ref{th:siinfpi}, cf. Section~\ref{sub:proofsipi}.
\findem


\subsection{\texorpdfstring{The pentagonal regions $\mathcal{R}_1$ and $\mathcal{R}_2$}{Pentagonal regions}}
Fix now $\si$ and $\tau$ in $\Dn$. We let $\mathcal{R}_1(\si,k)$ and $\mathcal{R}_2(\tau,m-3n-k+1)$ be the sets of {\em fillings} of the regions $\mathcal{R}_1$ and $\mathcal{R}_2$ that may arise as parts of FPL configurations $F$ verifying $\xpimk$ and such that $\si=\si(F)$ and $\tau=\tau(F)$ respectively. By reflecting $\mathcal{R}_2$ vertically, we see that $\mathcal{R}_2(\tau,m-3n-k+1)$ is the same region as $\mathcal{R}_1(\tau^*,m-3n-k+1)$, so we need only focus on $\mathcal{R}_1$.

\begin{figure}[!ht]
\includegraphics[width=\textwidth]{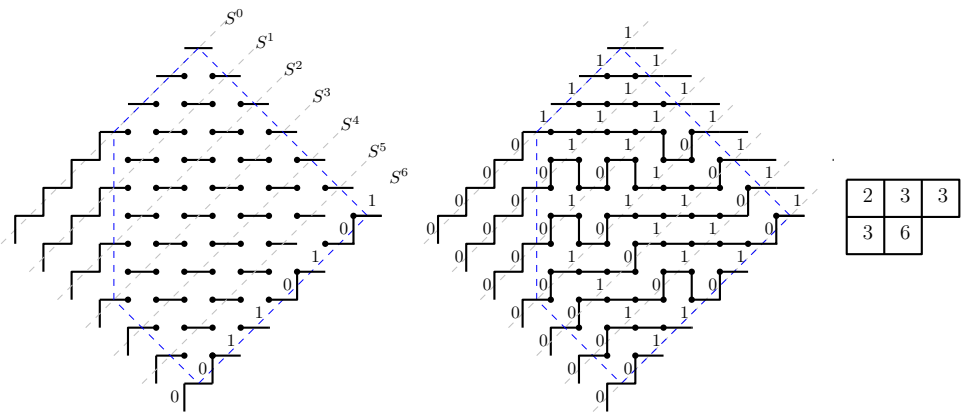}
\caption{Illustration of the bijection of Proposition~\ref{prop:regiontossyt}.
\label{fig:regiontossyt}}
\end{figure}

\begin{prop}\cite{CKLN}
\label{prop:regiontossyt}
Let $\si\in\Dn,$ and $k\geq 0$. The fillings in $\mathcal{R}_1(\si,k)$ are counted by $\ssyt(\si,n+k)$.
\end{prop}

\dem The proof given in~\cite{CKLN} proceeds in three steps: first one shows that FPL fillings are in fact in bijection with some lozenge tilings of a certain region; then these lozenge tilings are themselves in bijection with configurations of pairwise non intersecting lattice paths; and finally these configurations are in bijection with semistandard Young tableaux. We here give the direct bijection which is simply the composition of the previous three.

 As a preliminary step, extend the pentagonal region $\mathcal{R}_1$ by zigzag paths on the left to transform it in a (tilted) rectangular region; let $S^0,S^1,\ldots,S^{n+k}=S_1$ be the segments parallel to $S_1$ that divide the region, from the top left to the bottom right: see the left part of Figure~\ref{fig:regiontossyt}. Let us now fix a filling $f$ in $\mathcal{R}_1(\si,k)$: we encode the vertical edges from $f$ which cross the segment $S^i$ by a word $\si^i$ of length $2n$, as we did for $\si(F)$ in Definition~\ref{defi:sigmatau}: that is, these vertical edges correspond to $0$ while the other ones correspond to $1$. Then we have in fact $\si^i\in \Dn$ for all $i$, and the sequence $\zero=\si^0\leq \si^1\leq\ldots\si^{n+k}=\si$ is the desired semistandard tableau of shape $\si$ and length $n+k$. 
\findem

On  Figure~\ref{fig:regiontossyt} we obtain thus the following tableau:
\[
\mathbf{0}_4\leq 00001111\leq 00010111\leq 00101101\leq 00101101 \leq  00101101 \leq \si=00110101
\]

which can be represented compactly by the tableau on the right.


\subsection{\texorpdfstring{The triangular region $\mathcal{T}$}{Triangular region}}

With the help of the bijection of Proposition~\ref{prop:regiontossyt} and Figure~\ref{fig:regiontossyt}, we can identify fillings in $\mathcal{R}_1(\si,k)$ with tableaux of shape $\si$ and length $n+k$ for $\si\in \Dn$. We use this to give a diagram representing the decomposition of $A_\pi(m)$ in a compact manner on Figure~\ref{fig:thebigpicture_simple}.

\begin{figure}[ht]
\includegraphics[width=0.5\textwidth]{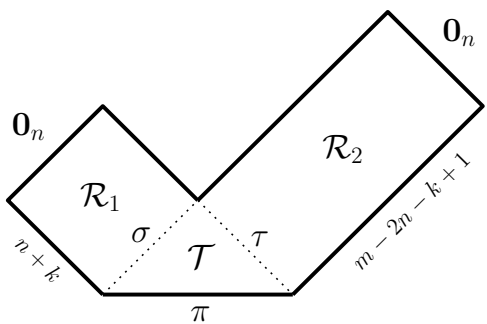}
\caption{Decomposition of $A_\pi(m)$.
\label{fig:thebigpicture_simple}}
\end{figure}

Now we consider the configurations in the triangle with left and right boundaries $\si$ and $\tau$ respectively, and bottom boundary $\pi$; we need to specify what constraints induced by the boundaries these configurations must obey. The choice of such boundaries correspond in fact to a certain link pattern on the external edges of the triangle, as pictured on the right of Figure~\ref{fig:thetriangle}. More precisely:
\begin{itemize}
 \item The $n$ occurrences of $1$ in the word $\si$ correspond to starting points of paths crossing the triangle, whose ending points correspond to the $n$ occurrences of $0$ in $\tau$.
\item The $n$ paths whose extremities are the lower external edges obey the matching $\pi$.
\item There may be closed paths inside the triangle.
\end{itemize}

Note that among the $n$ paths joining the left and right boundaries, the top path is reduced to a single vertex.

\begin{figure}[!ht]
\includegraphics[width=\textwidth]{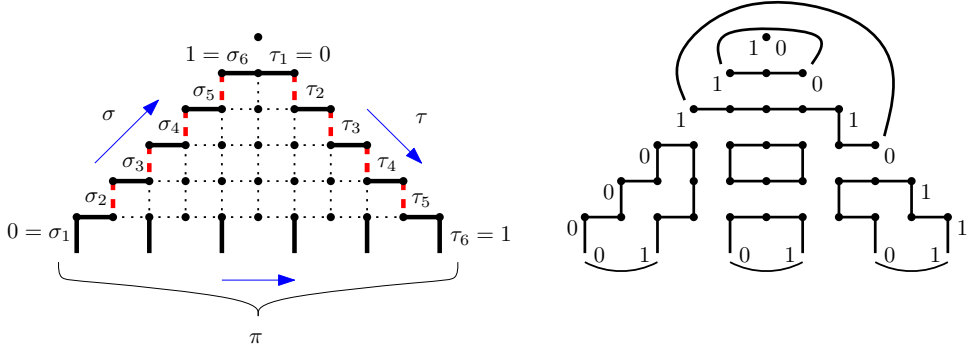}
\caption{The triangle $\mathcal{T}$, with a possible configuration in $\mathcal{T}_{000111,001011}^{010101}=
\mathcal{T}_{\mathbf{0}_3,001011}^{\mathbf{1}_3}$.
\label{fig:thetriangle}}
\end{figure}

\begin{defi}[TFPL configurations]
\label{defi:TFPL}
We define $\Tspt$ as the set of FPL fillings of the triangle $\mathcal{T}^n$ with left, bottom and right boundary conditions being given respectively by $\si,\pi$ and $\tau$. We let $\tspt=|\Tspt|$ be the number of such configurations, which we name TFPL configurations with boundary data ${\si,\pi,\tau}$.
\end{defi}


\subsection{\texorpdfstring{The polynomial formula for $A_\pi(m)$}{The polynomial formula }}

As is summarized by Figure~\ref{fig:thebigpicture_simple}, we can express the results from this section up to now as follows:

\begin{prop}
\label{prop:decomp_api}
Let $\pi$ be a matching of size $2n$, $m\geq 3n-1$ be an integer, and let $k$ be such that $0\leq k\leq m-3n+1$. There is an explicit bijection between \begin{enumerate}
\item Fully packed loop configurations with the link pattern $\xpimk$, and
\item Triplets $(P,f,Q)$ such that there exist $\si,\tau\in \Dn$ verifying:
\begin{itemize}
 \item $P$ is a semistandard Young tableau of shape $\si$ and length $n+k$;
\item $Q$ is a semistandard Young tableau of shape $\tau^*$ and length $m-2n-k+1$;
\item $f$ is a TFPL configuration in $\Tspt$.  
\end{itemize}                                                                                                                                                                                                             \end{enumerate}
Therefore we have the expression:
\begin{equation}
\label{eq:api}
A_\pi(m)=\sum_{\si,\tau\in \Dn}\ssyt(\si,n+k)\cdot \tspt \cdot\ssyt(\tau^*,m-2n-k+1),
\end{equation}
\end{prop}

This proposition was stated for $k=0$ in~\cite{CKLN}, and in general in~\cite{Thapper}. We can now state one of the main results of~\cite{CKLN} (which answered~\cite[Conjecture 6]{Zuber-conj}).

\begin{thm}[\cite{CKLN}]
\label{th:mainckln}
Let $\pi$ be a matching of size $2n$. The expression
\begin{equation}
\label{api:simple_expression} 
 A_\pi(m)=\sum_{\si,\tau\in \Dn}\ssyt(\si,m-2n+1)\cdot \tspt \cdot\ssyt(\tau^*,n),
\end{equation}

is valid for all $m\geq 0$.

 It is a polynomial in $m$ of degree $d(\pi)$ and leading coefficient $\frac{1}{H(\pi)}$.
\end{thm}

\dem The expression~\eqref{api:simple_expression} of $A_\pi(m)$ is the case $k=m-3n+1$ of~\eqref{eq:api}; the fact that it is also valid for $m<3n-1$ is proved in~\cite[Section 5]{CKLN}; it is obviously a polynomial. 

 From Theorem~\ref{th:siinfpi}(a) below, we know that nonzero terms in \eqref{api:simple_expression} occur only for $\si\leq \pi$,  and  by the remark following Theorem~\ref{th:hookcontent} we have that $\ssyt(\si,m-2n+1)$ has degree $d(\si)$ in $m$; therefore the polynomial $A_\pi(m)$ has degree at most $d(\pi)$. The coefficient in degree $d(\pi)$ is obtained when $\si=\pi$, which implies $\tau=\zero$ by Theorem~\ref{th:siinfpi}(b), and is thus given by 
\[
 \frac{1}{H(\pi)}\cdot t_{\pi,\zero}^\pi\cdot\ssyt(\zero,n)=\frac{1}{H(\pi)},
\]
which achieves the proof. 

\findem

In Sections~\ref{sect:linear} and~\ref{sect:tspt}, we will study the numbers $\tspt$ of TFPL configurations.

\section{\texorpdfstring{Linear recurrences for $A_\pi(m)$}{Linear recurrences}}
\label{sect:linear}

In this part we will follow the work of Thapper in~\cite{Thapper} motivated by the conjectures of Zuber concerning certain linear relations between quantities $A_\pi(m)$

\subsection{Proof of a conjecture of Thapper}

The next result was first stated as Conjecture~3.4 in~\cite{Thapper}. It is another property of the numbers $\tspt$, but of a different flavor than the ones in Section~\ref{sect:tspt} since it involves a relation between several of these numbers. 

\begin{thm}[{\cite[Conjecture 3.4]{Thapper}}]
\label{th:conjthapper}
Let $\si,\tau,\pi$ be elements of $\Dn$. Then we have the equality:
\begin{equation}
\label{eq:conjthapper}
\sum_{\stackrel{\si_1\in\Dn}{\si\to\si_1}}t_{\si_1,\tau}^\pi
=\sum_{\stackrel{\tau_1\in\Dn}{\tau^*\to\tau_1^*}}t_{\si,\tau_1}^\pi,\end{equation}
\end{thm}

\dem  The formula above can be better understood with the following diagrammatic representation\footnote{Thapper also defined the same representation, but did not seem to have noticed that the layer added on the left or right of the triangles corresponded precisely to horizontal or vertical strips.}, which uses the same conventions as Figure~\ref{fig:thebigpicture_simple}:
\begin{figure*}[!hb]
\begin{center}
\includegraphics[width=0.4\textwidth]{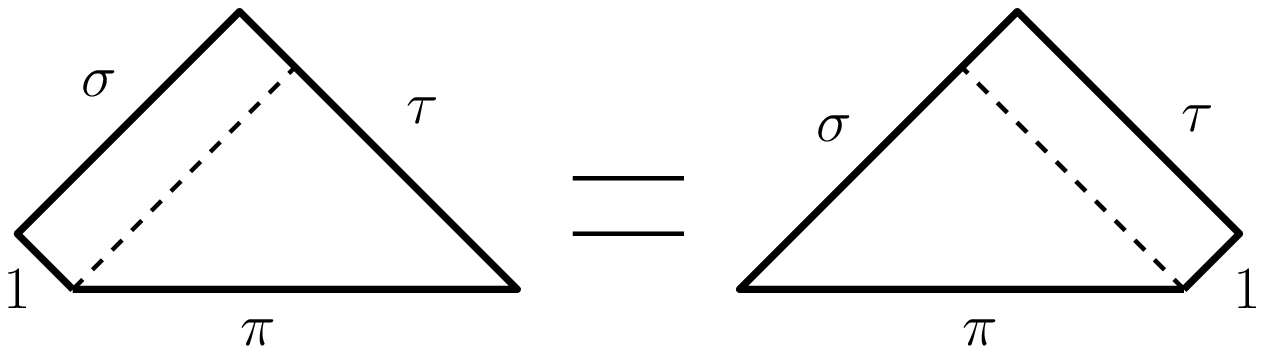}
\end{center}
\end{figure*}

Indeed the l.h.s. of~\eqref{eq:conjthapper} counts configurations in a ``left extended'' triangle, and the r.h.s. counts configurations in a ``right extended'' triangle, both with boundaries $\si,\tau,\pi$. We will prove the result by showing that configurations counted by each member are in bijection.

We fix first $k,m>0$ integers such that $m>3n+k-1$, as well as semistandard Young tableaux $P$ of shape $\si$ and length $n+k$, and a tableau $Q$ of shape $\tau^*$ and length $m-2n-k$. Such tableaux exist if $k$ is chosen big enough, which we also suppose.

Now let $f$ be a left-extended TFPL configuration with boundaries $\si,\tau,\pi$. By the bijection of Proposition~\ref{prop:decomp_api}, $f,P$ and $Q$ define a certain FPL configuration $F$ on the grid $G_{n+m}$, with link pattern $\xpimk$; see Figure~\ref{fig:thapper_conj_diagram2}, left.

We now apply Wieland's rotation $W^{-1}$ to $F$ (cf. Theorem~\ref{th:wieland}), and we obtain a certain FPL configuration $F'$ with link pattern $X(\pi,m,k-1)$.  By Proposition~\ref{prop:decomp_api} again, this is equivalent to the data of a right extended TFPL configuration $f'$ with left and right boundaries $\si'$ and $\tau'$, together with two tableaux $P',Q'$ of respective lengths $n+k,m-2n-k$, and respective shapes $\si',\tau'$. We can represent this in the diagram:

\begin{figure}[!ht]
\begin{center}
\includegraphics[width=0.8\textwidth]{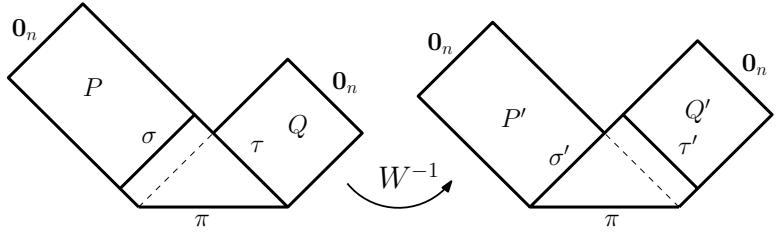}
\end{center}
\caption{Action of $W^{-1}$ on $F$.
\label{fig:thapper_conj_diagram2}
}
\end{figure}

The following lemma shows that the action of Wieland's rotation in the pentagonal regions is essentially trivial.

\begin{lemma}
We have $P=P'$ and $Q=Q'$. In particular, $\si=\si'$ and $\tau=\tau'$.
\end{lemma}

\demof{the lemma}
It is enough to prove the first equality, since the second follows essentially by applying $W^{-1}$ and reflecting the picture. 
 We need to translate back the claim in terms of fillings: see Proposition~\ref{prop:regiontossyt} and Figure~\ref{fig:regiontossyt}, from which we borrow notations. Let $\mathcal{R}$ be the region of $G_{n+m}$ relative to $P$: this is the region $\mathcal{R}_1$ relative to $F$ without the slice between $S^{n+k-1}$ and $S^{n+k}$.The region  $\mathcal{R}'$ corresponding to the tableau $P'$ is then exactly the region $\mathcal{R}$ shifted one step south. Then the equality $P=P'$ means that a vertical edge $v_1$ crosses a certain segment $S^i$  in the region $\mathcal{R}$ if and only if, in the region $\mathcal{R}'$ of $F'$, there is a vertical edge $v_2$ one step south of the original $v_1$ . We thus need to study how the transformation $W^{-1}$ acts in $\mathcal{R}$, which boils down to a local analysis:
\begin{center}
\includegraphics[width=0.7\textwidth]{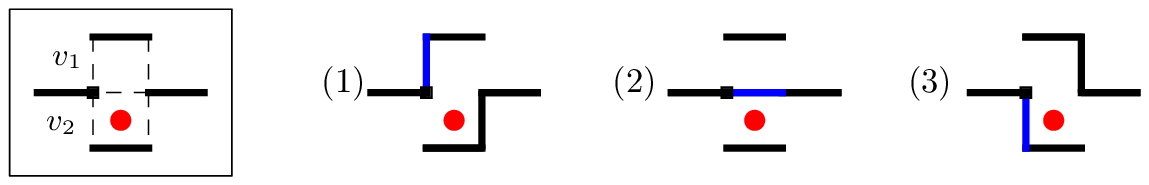}
\end{center}
Here is how to read this: on the left we pictured a local view of $\mathcal{R}$ around an edge $v_1$ defined as above. In the active cell (marked with a dot), we have to perform the local operation $U$ as part of the transformation $W^{-1}$ (cf. proof of~Theorem~\ref{th:wieland} for the definitions). We have to prove that $v_1$ belongs to $F$ if and only if $v_2$ belongs to $F'=W^{-1}(F)$. We distinguish three cases, which correspond to the three possibilities for the second edge attached to the marked vertex:
\begin{enumerate}
\item $v_1$ is in $F$ and $v_2$ is not: since the marked cell is not surrounded by two parallel edges, the action of $U$ will cause $v_2$ to belong to $F'$;
\item neither $v_1$ nor $v_2$ belongs to $F$: here the marked cell is surrounded by parallel edges, so $U$ will leave it invariant and $v_2$ does not belong to $F'$;
\item $v_2$ belongs to $F$ while $v_1$ does not: since the marked cell is not surrounded by parallel edges $v_2$ belongs to $F'$. 
\end{enumerate}

This covers all possible cases of edges attached to the marked vertex, and in each case we have that $v_1$ belongs to $F$ if 
and only if $v_2$ belongs to $F'$. Therefore $P=P'$ and their shapes $\si$ and $\si'$ coincide, which achieves the proof of the lemma.

\findem 

We can now finish the proof of Theorem~\ref{th:conjthapper}: by the previous lemma, $\si'=\si$ and $\tau'=\tau$. Therefore the restriction of $W^{-1}$ to left-extended triangles gives the wanted bijection with the right extended triangles preserving the boundary conditions, which proves bijectively the relation~\eqref{eq:conjthapper}.\findem 

\begin{rem}
It is of course also possible to prove the theorem by studying the effect of $W$ directly on extended triangles. But this would boil down essentially to copying the proof of Wieland, while here we just needed his result in the case of the grid. Furthermore, it is a nice byproduct of the proof that in regions of the type of $\mathcal{R}_1$, Wieland's rotation operates by a simple shift. 
\end{rem}

\subsection{\texorpdfstring{$A_\pi(m)$ as a linear combination of certain $A_\alpha(m-1)$ }{Linear combinations}}

 We follow here the main ideas of~\cite{Thapper}, by expressing some of the results in the matrix language. All matrices will have rows and columns indexed by $D_n$ (with a given linear order). We also consider such matrices as linear endomorphisms of the vector space $\mathbb{C}\Dn$ of formal complex linear combinations of elements of $\Dn$, by setting $g(\tau):=\sum_{\si}g_{\si\tau}\tau$ for any matrix $g$.

\begin{defi}[Matrices $\b,\b^*,\t,\t^{\pi}$]
Given $\si,\tau,\pi\in \Dn$, we define the following matrix elements:
\begin{align*} 
\b_{\si\tau}&=\begin{cases}
               1\quad \text{if}\quad \tau/\si \quad\text{is a horizontal strip},\\
               0 \quad\text{otherwise};
              \end{cases}\\
 \wb_{\si\tau}&=\begin{cases}
                1 \quad \text{if}\quad \si^*/\tau^* \quad\text{is a horizontal strip,}\\
                0 \quad\text{otherwise};
                \end{cases}\\
(\t^{\pi})_{\si\tau}&=\tspt;\\
\t_{\si\pi}&=t_{\si,\zero}^\pi.
\end{align*}
 \end{defi}

Now if $k$ such that  $0\leq k\leq m-3n+1$, Equation~\eqref{eq:api} can be rewritten:
\begin{equation}
\label{eq:apimatrix}
A_\pi(m)=\left(\b^{n+k}\t^{\pi}\wb^{m-2n-k+1}\right)_{\zero\zero},
\end{equation}

and this can be visualized also on Figure \ref{fig:thebigpicture_simple}. Also, 
the result of Theorem~\ref{th:conjthapper} is equivalent to:
\begin{equation}
\label{eq:commu}
\b\t^{\pi}=\t^{\pi}\wb\quad\text{for any }\pi\in\Dn.
\end{equation}

 which is in fact the way it is formulated in~\cite[Conjecture 3.4]{Thapper}. Thanks to repeated uses of the relation~\eqref{eq:commu}, we can push $t^{\pi}$ to the right in \eqref{eq:apimatrix} to get
\[A_\pi(m)=\left(\b^{m-n+1}\t^{\pi}\right)_{\zero\zero}\quad\text{for }m\geq 3n-1
\]

Now the expression for $A_\pi(m)$ above\footnote{let us note that this expression can be deduced from the FPL configuration by using Wieland's rotation. Indeed, when $k$ becomes greater than $m-3n+1$, the triangle $\mathcal{T}^n$ gets more and more truncated by the right border of the grid, up until $k=m-2n+1$ where the truncated triangle  is equivalent to a triangle $\Tspt$ with $\tau=\zero$, and we get the formula.} is polynomial in $m$ and we can in fact write the following proposition:

\begin{prop}
\label{prop:apireduced}
For all integers $m$, we have  $A_\pi(m)=\left(\b^{m-n+1}\t\right)_{\zero\pi}$.
\end{prop}

The passage from $\t$ to $\t^\pi$ is validated by the fact that $\left(\t^\pi\right)_{\si\zero}=\left(\t\right)_{\si\pi}$ by definition. By Theorem~\ref{th:siinfpi} the coefficients $\left(\t\right)_{\si\pi}=t_{\si,\zero}^\pi$  are integers, equal to $0$ unless $\si\leq \pi$, and such that $\left(\t\right)_{\pi\pi}=1$. This means that, if we give the basis $\Dn$  a linear order extending $\leq$, then the matrix of $\t$ becomes upper triangular with ones on its diagonal; it is thus invertible with inverse $\t^{-1}$ being also triangular with ones on its diagonal, and has integer entries.

\begin{defi}[Matrix $\c$]
\label{def:c}
For a given $n$ we define the endomorphism $\c$ by
\[\c:=\t^{-1}\b\t.\]
\end{defi}

We can now state the result conjectured\footnote{  Thapper shows in fact that the result is a consequence of the conjectural relations $\b\ttau=\ttau\c$ for any $\tau$, where  $\left(\ttau\right)_{\si\pi}=\tspt.$ We show here that one can obtain the result without the help of these relations.} by Thapper~\cite[Proposition 3.5]{Thapper}:

\begin{thm}
\label{th:linearcomb}
For any $\pi\in\Dn$, we have the polynomial identity:
\[A_\pi(m)=\sum_{\alpha\in\Dn}\c_{\alpha\pi}A_\alpha(m-1).\]
\end{thm}

\dem By Proposition~\ref{prop:apireduced} and the definition of $\c$, we get
\begin{align*}
 A_\pi(m)&=\left(\b^{m-n+1}\t\right)_{\zero\pi}
         =\left(\b^{m-n}\t\c\right)_{\zero\pi}\qquad (\text{because }\b\t=\t\c)\\
         &=\sum_{\alpha\in\Dn}\left(\b^{m-n}\t\right)_{\zero\alpha}\c_{\alpha\pi}=\sum_{\alpha\in\Dn}A_\alpha(m-1)\c_{\alpha\pi},
\end{align*}

from which the result follows, again by Proposition~\ref{prop:apireduced}.\findem

Note that the fact there exist such coefficients is not in itself surprising: $A_\pi(m)$ has degree $d(\pi)\leq n(n-1)/2$ in $m$, and Theorem \ref{th:linearcomb} expresses it as a linear combination of many more polynomials in general Furthermore such coefficients are in fact not unique as can be easily checked on examples.

What makes these coefficients interesting is the following: first, we have an explicit definition for them (even though it is not immediate to extract a lot of information from it). Then, a look at the data shows a lot of (mostly conjectural) properties for them, which lead then to new conjectures by Theorem~\ref{th:linearcomb}: an example of this are the Conjectures \ref{conj:capconj} and \ref{conj:capconj2} below.


\subsection{\texorpdfstring{Computation of some coefficients $c_{\alpha\pi}$}{The linearization coefficients}}

The definition of $\c$ can be rewritten in the following equations for the coefficients $c_{\alpha\pi}$: 
\begin{equation}
\label{eq:rec_calphapi}
c_{\alpha\pi}=
\begin{cases}
0\quad\text{if}\quad\alpha\nleq\pi;\\
             1\quad\text{if}\quad\alpha=\pi;\\ 
             \sum\limits_{\stackrel{\beta\in\Dn}{\alpha\to\beta,\beta\neq\alpha}}t_{\beta,\zero}^{\pi}
-\sum\limits_{\stackrel{\gamma\in\Dn}{\alpha<\gamma<\pi}}
c_{\gamma\pi}t_{\alpha,\zero}^{\gamma}\quad\text{if}\quad\alpha<\pi.
\end{cases}
\end{equation}

These relations clearly characterize completely the coefficients $c_{\alpha\pi}$, and we illustrate this by computing some of  these coefficients:

\begin{prop}
Let $\alpha\leq\pi$, and consider the skew shape $\pi/\alpha$:
\begin{enumerate}
\item if $\pi/\alpha$ consists of one cell, or two disconnected cells, then $c_{\alpha\pi}=1$;
\item if  $\pi/\alpha$ consists of $a\geq 2$ cells in the same row, or $a\geq 2$ cells in the same column, then $c_{\alpha\pi}=0$.
\end{enumerate}
\end{prop}

\dem $(1)$ If $\pi/\alpha$ is just one cell, then the formula~\eqref{eq:rec_calphapi} gives just one term $t_{\pi}^{\pi}$, which is equal to $1$ by Theorem~\ref{th:siinfpi}. 

If $\pi/\alpha$ has two disconnected cells, say a first one in column $h+1$ and a second one in column $h'+1$ with $h'>h$, then let $\pi_1$ (resp. $\pi_2$) denote the element of $\Dn$ obtained by adding the first cell (resp. the second one) to $\alpha$. Then the recurrence formula gives 
\[c_{\alpha\pi}=t_{\pi_1}^\pi+t_{\pi_2}^\pi+t_{\pi}^\pi
-c_{\pi_1\pi}t_{\alpha}^{\pi_1} - c_{\pi_2\pi}t_{\alpha}^{\pi_2}.\]
 Now from the case of one cell we have $c_{\pi_1\pi}=c_{\pi_2\pi}=1$, while $t_{\alpha}^{\pi_1}=t_{\pi_2}^\pi$, $t_{\alpha}^{\pi_2}=t_{\pi_1}^\pi$ follow from Corollary~\ref{cor:rotatedpart}. The only remaining term after the simplifications is then $t_{\pi}^\pi=1$.

$(2)$ Let $\pi_0=\alpha<\pi_1<\ldots<\pi_{a-1}<\pi_a=\pi$ denote the shapes between $\alpha$ and $\pi$. Define also $X=2n-2-h$, where $h+1$ is the column of the cell $\pi_1/\alpha$. 

 $\bullet$ If $\pi/\alpha$ consists of cells in the same column, the result holds for $a=2$, since one gets
\[c_{\alpha\pi}=t_{\pi^1}^{\pi}-t_{\pi^0}^{\pi^1}=X-X=0.\]
 
 For $a>2$, only $\pi_1$ appears in the first sum in Equation \ref{eq:rec_calphapi}; by induction on $a$, it simplifies to \[c_{\alpha\pi}=t_{\pi_1}^{\pi_a}-t_{\alpha}^{\pi_{a-1}}=0,\] 

the two quantities being equal by Corollary \ref{cor:rotatedpart}.

 $\bullet$ If cells are in the same row, one has for the case $a=2$:
 \[c_{\alpha\pi}=t_{\pi^1}^{\pi}+t_{\pi}^{\pi}-t_{\pi^0}^{\pi^1}=(X-1)+1-X=0.\]
 
  and proceeding by induction on $a$,  Equation~\eqref{eq:rec_calphapi} becomes
\[c_{\alpha\pi}=\sum_{k=1}^a t_{\pi^0}^{\pi^k}-t_{\pi^0}^{\pi^{a-1}}=\sum_{k=1}^a \binom{X-k}{a-k}-\binom{X}{a-1}=0,\]
where here also we used Corollary \ref{cor:rotatedpart}.
\findem

An interesting conjecture from~\cite{Thapper} claims that $c_{\alpha\un}=1$ for all $\alpha\in\Dn$; by Theorem~\ref{th:linearcomb}, this implies the following polynomial identity:
\begin{equation}
 A_{\un}(m)=\sum_{\pi\in\Dn} A_\pi(m-1).
\end{equation}
This last equality was conjectured by Wieland~\cite{Wieland}, and is now a theorem thanks to the Razumov--Stroganov correspondence~\cite{ProofRS} and Equation~(4.8) in~\cite{artic47}. 

We conjecture the following special values for $c_{\alpha\pi}$:

\begin{conj}
\label{conj:capconj}
 Suppose $\pi=(01)^i0^{\ell+1} 1^{\ell+1} (01)^j=\mathbf{1}_i\mathbf{0}_{\ell+1}\mathbf{1}_j$. Then 
\[
 c_{\alpha\pi}=\begin{cases}
                1\quad&\text{if $\alpha$ has the form }\alpha=u\mathbf{0}_{\ell}v\\
                &\text{where }|u|=2i+1,|v|=2j+1;\\
                0\quad&\text{otherwise.}
               \end{cases}
\]
\end{conj}

When $\ell=0$ this reduces to Thapper's conjecture that $c_{\alpha\un}=1$  for any $\alpha\in\Dn$.\\
By Theorem~\ref{th:linearcomb}, the truth of Conjecture~\ref{conj:capconj} implies a certain relation between FPLs on the square grid. To express this relation, we represent the numbers $A_X$ for a link pattern $X$ by using the chord diagram notation. The conjecture we obtain is the following:

\begin{conj}
\label{conj:capconj2} For any $i,j,\ell,m\geq 0$, we have the equality: 
\begin{center}
 \includegraphics{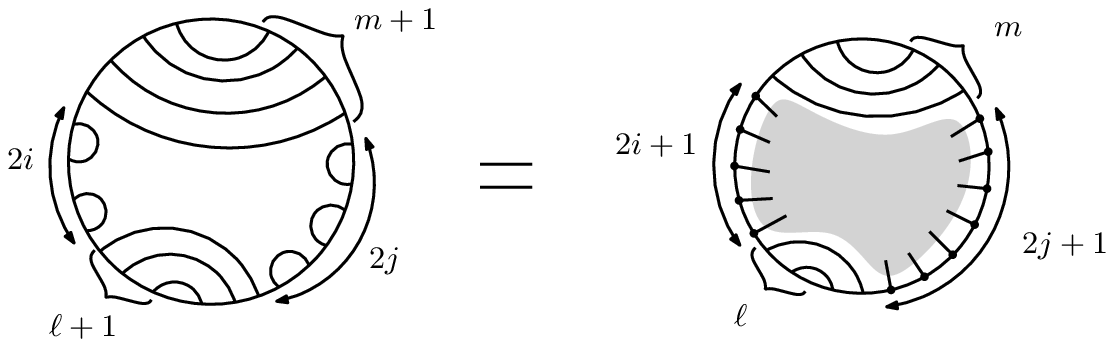}
\end{center}
where the right member represents a sum over $C_{i+j+1}$ chord diagrams.
\end{conj}

This has been checked for all $i+j+\ell\leq 9$ and any $m$ by Tiago Fonseca (personal communication).

\section{\texorpdfstring{Some properties of the TFPL numbers $\tspt$}{The TFPL numbers}}
\label{sect:tspt}

The first obvious property is the vertical symmetry of the triangle $\mathcal{T}$, which we actually used on several occasions already: with boundaries $\si,\tau,\pi$ in $\Dn$, we have that
\[\tspt=t_{\tau^*,\si^*}^{\pi^*}.\]

Another property, deduced by Thapper from Theorem~\eqref{eq:api}, is the following

\begin{prop}[\cite{Thapper}]
\label{prop:siplustauinfpi}
We have $\tspt=0$ unless $\deg(\sigma)+\deg(\tau)\leq \deg(\pi)$.
\end{prop}

In a forthcoming paper~\cite{NadFPL2}, we will enumerate configurations in $\Tspt$ for which we have equality in the proposition above: the answer turns out to be given by {\em Littlewood--Richardson coefficients}.\medskip 

In the rest of this section, we will prove some more properties. The first one is given in Theorem~\ref{th:siinfpi}; this property was first proved in \cite{CKLN} and was the key ingredient in determining the leading term of the polynomials $A_\pi(m)$. We will give here a much shorter and more illuminating proof. The second set of properties concerns the quantities $\tspt$ when the words $\si$ and $\pi$ have common prefixes and suffixes, and is given in Theorems \ref{th:commonps} and \ref{th:rotatedpart}.

The proofs of this section will rely heavily on the introduction of \emph{oriented} TFPL configurations, which are a natural superset of TFPLs which we proceed to describe


\subsection{Oriented configurations}
\label{sub:oriented}

We fix coordinates on the triangle $\mathcal{T}^n$, by letting the origin be the bottom left vertex of the triangle, and take as basis vectors $(1,0)$ and $(0,1)$. The vertices of $\mathcal{T}^n$ are the points of coordinates $(x,y)\in \mathbb{Z}^2$ which verify $x\geq y \geq 0$ and $x+y\leq 4n-2$. Such vertices can be partitioned in lines: for $i\in\{1,\ldots,2n\}$, we define $E_i$ as the vertices of $\mathcal{T}^n$ such that $x+y=2i-2$, and for $i\in \{1,\ldots,2n-1\}$, we define $O_i$ as the vertices of $\mathcal{T}^n$ such that $x+y=2i-1$. The case $n=3$ is illustrated on Figure~\ref{fig:the_triangle2}.

\begin{figure}[!ht]
\begin{center}
\includegraphics{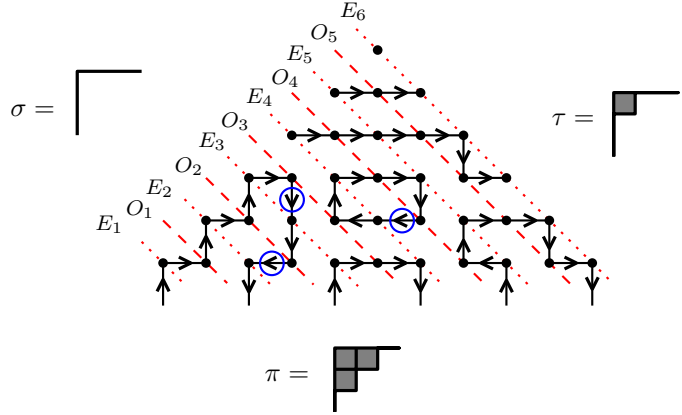}
\caption{A canonically oriented configuration, corresponding to the TFPL configuration from Figure~\ref{fig:thetriangle}.
\label{fig:the_triangle2}}
\end{center}
\end{figure}

Now let us suppose we have boundary configurations $\si,\tau,\pi$ in $\Dn$. We define an orientation for all edges around the triangle as follows. On the left boundary, we orient edges to the right and upwards; on the right boundary, we orient them to the right and downwards. Finally, for the $2n$ vertical external edges on the bottom, we orient the one attached to $(2i-2,0)$ upwards if $\pi_i=0$, and downwards if $\pi_i=1$, for $i\in\{1,\ldots,2n\}$.

We then define an \emph{oriented TFPL configuration} as an oriented subgraph of the triangle such that each internal vertex has one incoming edge and one outgoing edge; by internal we exclude here the vertices of the form $(i,i)$ and $(2n-1+i,2n-1-i)$ occurring on the left and right boundaries. We let $\oTspt$ be the set of oriented TFPLs with  boundaries $\si,\tau,\pi$.

 Now if we are given a (usual) TFPL configuration $f$ in $\Tspt$, we define a canonical orientation to each of its edges, as follows: For each non closed path, there is clearly a unique way to define the orientation, while if such a path is closed, we orient it clockwise by convention. In this way we associate to each configuration $f\in \Tspt$ an oriented configuration that we will denote by $or(f)$: see an example on Figure~\ref{fig:the_triangle2} which represents the canonical orientation of the TFPL configuration of Figure~\ref{fig:thetriangle}  Such oriented configurations will be called \emph{well-oriented}. 

Therefore for any $\si,\pi,\tau\in\Dn$, we have a natural \emph{injection from $\Tspt$ to $\oTspt$}.

\begin{rem}
\label{rem:specialcase}
We make two remarks about these configurations: first, vertices $(i-1,i-1)$ where $i$ is such that $\si_i=1$ have one outgoing edge but no incoming edge, while all other vertices of $\mathcal{T}^n-E_{2n}$ have exactly one outgoing and one incoming edge. Then, external bottom edges are attached to only one vertex of $\mathcal{T}^n$, while all other edges join two vertices of the triangle.
\end{rem}

We will now use oriented configurations to give a new proof of the following theorem:

\begin{thm}[\cite{CKLN}]
\label{th:siinfpi}
Let $\si,\pi,\tau$ be in $\Dn$. Then

(a) $\tspt= 0$ unless $\si\le \pi$.

(b) If $\si=\pi$, then $t_{\pi,\zero}^\pi=1$ and $t_{\pi,\tau}^\pi=0$ for $\tau\neq\zero$.
\end{thm}

\subsection{Proof of Theorem~\ref{th:siinfpi}}
\label{sub:proofsipi}

We first need the following definition:

\begin{defi}[$\mathcal{N}_i(f)$ and $N_i(f)$]
Let $\si,\tau,\pi$ be in $\Dn$, $f$ an oriented TFPL in $\oTspt$, and $i$ be an integer in $\{1,\ldots,2n-1\}$. We define $\mathcal{N}_i(f)$ as the set of oriented edges in $f$ which are directed from a vertex in $O_i$ to a vertex in $E_i$. We also define $N_i(f)=|\mathcal{N}_i(f)|$.
\end{defi}

These oriented edges are circled in the example of Figure~\ref{fig:the_triangle2}: we have $N_i(f)=0,1,1,1,0$ for $i=1,2,3,4,5$ respectively. We can now state the key lemma:

\begin{lemma}
\label{lem:Ni_f}
Let $\si,\tau,\pi$ be in $\Dn$, $f$ a configuration in $\oTspt$. Then 
\begin{equation}
\label{eq:green}
N_i(f)-N_{i-1}(f)=\pi_i-\si_i,\qquad\text{for~}i=1,\ldots,2n-1,
\end{equation}
where $N_0(f)=0$ by convention.
\end{lemma}

\demof{Lemma~\ref{lem:Ni_f}}
Consider the $i$ vertices of the line $E_i$: they all have an incoming edge, except $(i-1,i-1)$ when $\si_i=1$. If this incoming edge comes from $O_i$ it is an element of $\mathcal{N}_i(f)$; let $X_i(f)$ be the other incoming edges, and $x_i(f):=|X_i(f)|$. We have thus 
\begin{equation}
\label{eq:eq1}
N_i(f)+x_i(f)+\si_i=i.
\end{equation}
Similarly, consider the $i-1$ vertices on the line $O_{i-1}$; then each of them has exactly one outgoing edge, and if this edge goes to the line $E_{i-1}$ it is a member of $\mathcal{N}_{i-1}(f)$. We form the set $Y_i(f)$ with the other outgoing edges, and let  $y_i(f):=|Y_i(f)|$. We obtain here
\begin{equation}
\label{eq:eq2}
N_{i-1}(f)+y_i(f)=i-1.
\end{equation}

Now the sets $Y_i(f)$ and $X_i(f)$ consist of the same edges, namely those directed from $O_{i-1}$ to $E_i$, up to an extra edge in $X_i$ in the case $\pi_i=0$ which is the external edge incoming in $(2i-2,0)\in E_i$. Therefore we get the relation $x_i(f)=y_i(f)+(1-\pi_i)$ and deduce the proposition from it together with Equations~\eqref{eq:eq1} and~\eqref{eq:eq2}.

\findem

Given $j\in\{1,\ldots,2n\}$, we sum the relations~\eqref{eq:green} for $i$ going from $1$ to $j$ and obtain $\si_{\leq j}-\pi_{\leq j}=\sum_{i=1}^j N_i(f)$, which is nonnegative. As we noticed in Section~\ref{sub:Dn}, this proves that $\si\leq \pi$. This is valid for all oriented configurations, and thus in particular for well oriented configurations $or(f)$ for $f\in \Tspt$, which finishes the proof of Theorem~\ref{th:siinfpi}(a). 

We note also that the proof does not use the fact that $\si\in\Dn$, and that we obtain this as a byproduct of $\si\leq \pi$. This completes thus the proof of Proposition~\ref{prop:siDn}.

To prove part (b) of Theorem~\ref{th:siinfpi}, one needs to show that if $\si=\pi$, there is a unique possible configuration in the triangle $\mathcal{T}^n$, and moreover such a configuration has $\tau=\zero$ on the right boundary. This is easy and was done at the end of Section 7 in~\cite{CKLN}; this is also a very special case of Theorem~\ref{th:rotatedpart} below.

\subsection{Common prefixes and suffixes}

We now come to our last properties of the numbers $\tspt$. 

\begin{thm}
\label{th:commonps}
Let $\pi,\si,\tau\in \Dn$, and suppose that there exist  $01$-words $u,\si',\pi',v$ such that the following factorizations hold:
\begin{equation}
\label{eq:commonps}
\si=u\si'v\quad\text{and}\quad\pi=u\pi'v.
\end{equation}
Let $n-a=|u|_0+|v|_0$ and $n-b=|u|_1+|v|_1$. Then $\tspt=0$ unless $\tau$ is of the form:
\[ \tau=0^{n-a}\tau'1^{n-b}. \]
for a certain word $\tau'$.
\end{thm}

We defer the proof of this theorem to Section~\ref{sect:proofcommonps}. This theorem is perhaps better understood by looking at Figure~\ref{fig:sipidiagram}: it means that if the skew shape $\pi/\si$ fits into a rectangle $R$, then $\tspt\neq 0$ implies that the Young diagram of $\tau$ fits into $R$. 

\begin{figure}[!ht]
\includegraphics[width=0.6\textwidth]{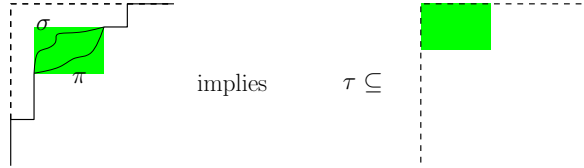}
\caption{Diagram representation of Theorem~\ref{th:commonps}.
\label{fig:sipidiagram}}
\end{figure}
Now suppose that the conditions of Theorem~\ref{th:commonps} are verified, so that we have the factorizations $\si=u\si'v$, $\pi=u\pi'v$ and $ \tau=0^{n-a}\tau'1^{n-b}$, with $n-a=|u|_0+|v|_0$ and $n-b=|u|_1+|v|_1$. We have in particular that $\si',\pi'$ and $\tau'$ are $01$-words of length $2n-|u|-|v|=a+b$. More precisely,  $|\si'|_0=|\tau'|_0=a$ and $|\si'|_1=|\tau'|_1=b$.

\begin{defi}
With the hypotheses above, we say that $\si$ and $\pi$ form a \emph{reverse shape} if we have in addition that $\pi'=1^b0^a$.
\end{defi}

The terminology was chosen because the corresponding skew diagram $\pi/\si$ is (up to translation) a Ferrers diagram rotated by $180^\circ$.

In the case of a reverse shape for $\si$ and $\pi$, we can go further than Theorem~\ref{th:commonps} and actually provide a determinantal formula for the coefficients $\tspt$. For a $01$-word $u=u_1\cdots u_p$ and $\eps\in\{0,1\}$, define $I_\eps(u)$ to be the set of indices $i\in\{1,\ldots,p\}$ such that $u_i=\eps$.

\begin{thm}
\label{th:rotatedpart}
Suppose $\si,\pi$ form a reverse shape, so that we have the factorizations $\si=u\si'v$, $\pi=u1^b0^av$. Assume moreover that $\tau=0^{n-a}\tau'1^{n-b}$ since otherwise $\tspt=0$ by Theorem~\ref{th:commonps}. Define the $a\times a$ matrix $M_1$ by 

\[M_1=\left(\binom{2n-|u|_1-2-(a+b)+j_e}{i_d+j_e-(a+b+1)}\right)_{d,e=1,\ldots, a}\] 

where $I_1(\si')=\{i_1<i_2<\ldots<i_a\}$ and $I_1(\tau')=\{j_1<j_2<\ldots<j_a\}$. Define also the $b\times b$ matrix $M_0$ by 

\[M_0=\left(\binom{2n-|u|_1-1-i'_d}{(a+b+1)-i'_d-j'_e}\right)_{d,e=1,\ldots, b}\]

where  $I_0(\si')=\{i'_1<i'_2<\ldots<i'_b\}$ and $I_0(\tau')=\{j'_1<j'_2<\ldots<j'_b\}$.
 Then \[\tspt=det(M_0)=det(M_1).\]
\end{thm}

 Note that when $a=1$ or $b=1$, then $M_0$ or $M_1$ is a $1\times 1$ matrix. We thus get immediately

\begin{cor}
\label{cor:rotatedpart}
If $\si=u0^b1v$, and $\pi=u10^bv$, then $\tspt$ is zero unless $\tau=0^{n-k}10^{k}1^{n-1}$ for a certain $k\in \{0,\ldots,b\}$, and in this case $\tspt=\binom{2n-|u|_1-2-k}{b-k}$.

If $\si=u01^av$, and $\pi=u1^a0v$, then $\tspt$ is zero unless $\tau=0^{n-1}1^{k}01^{n-k}$ for a certain $k\in \{0,\ldots,a\}$, and in this case $\tspt=\binom{2n-|u|_1-2-a}{a-k}$.
\end{cor}

We defer also the proof of Theorem~\ref{th:rotatedpart} to Section~\ref{sect:proofcommonps}, since it relies heavily on the proof of Theorem~\ref{th:commonps}.

\section{Proof of Theorems~\ref{th:commonps} and~\ref{th:rotatedpart}}
\label{sect:proofcommonps}

We have first the following lemma, which is a variant of de Gier's original one in~\cite{degierloops}:

\begin{lemma}
\label{lem:degierplus}
Suppose we have a vertex $v$ in an oriented configuration $f$, such that, among its four adjacent vertices in $\mathcal{T}^n$, three have an incoming edge which does not come from $v$. Then there is an edge in $f$ from $v$ to the fourth vertex. 
\end{lemma}

\dem The proof of this lemma is shorter than its statement: the outgoing edge from $v$ can indeed only be going to  the fourth vertex. \findem

Despite its simplicity, careful applications of this lemma are basically all that is needed to prove Theorem~\ref{th:commonps}. 

We will first deal with common suffixes, then with the prefixes, and finally merge the two cases together. We will reason on {\em oriented configurations}, and we use the notations $E_i,O_i$ from Section~\ref{sub:oriented}.


\subsection{Common suffix} We will reason by induction on length of the suffix $v$, and show the following:

\begin{lemma}
\label{lem:commons}
For any configuration in $\Tspt$, then for all $i\in\{2n-|v|+1,\ldots,2n\}$, all points of the line $E_i$ have an incoming edge coming from the line $O_{i-1}$, and this edge is:
\begin{itemize}
\item vertical, and thus oriented upwards,  if $\si_i=\pi_i=0$.
\item horizontal, and thus oriented to the right, if $\si_i=\pi_i=1$.
\end{itemize}
\end{lemma}

\demof{Lemma~\ref{lem:commons}} This is proved by induction on $2n-i$. For $i=2n$, we have by definition that there is an horizontal edge from $O_{2n-1}$ to $E_{2n}$ for every point of $E_{2n}$ except the topmost one, and indeed one has $\si_{2n}=\pi_{2n}=1$; the base case is thus proved. 

Now suppose the statement holds for $i+1$; one has then that all vertices of $O_{i-1}$  have an edge directed (up or right) towards $E_i$, and in particular \emph{the edges incoming to $E_{i-1}$ cannot come from $O_{i-1}$}. We consider two cases:

{$\bullet\,\si_i=\pi_i=0$:} Since $\si_i=0$, we have an up edge $e_1$ between $(i,i-1)\in O_{i-1}$ and $(i,i)\in E_i$. Now look at the point $(i+1,i-1)\in E_i$: its left,top and right neighbor all have outgoing edges which are not directed towards it, and therefore by Lemma~\ref{lem:degierplus} there is an up edge  $e_2$ between $(i+1,i-2)\in O_{i-1}$ and $(i+1,i-1)\in E_i$. We can then reiterate the same argument for the points $(i+k,i-k)\in E_i$ for $k=2,\ldots,i$ successively, which imply the existence of the up edges below $(i+k,i-k)\in E_i$, and the lemma is proved in this case. \footnote{note that in particular that we get for $k=i$ that the external edge attached to $(2i,0)$ is incoming, i.e. $\pi_i=0$. So we showed that $\si_i=0$ implies $\pi_i=0$, which is a special case of Theorem~\ref{th:siinfpi}.}

{ $\bullet\,\si_i=\pi_i=1$:} Now $\pi_i=1$ means that the external edge attached to $(2i,0)$ is outgoing. By Lemma~\ref{lem:degierplus}, the only possibility for the edge incoming in $(2i,0)$ is the right oriented edge $(2i-1,0)\rightarrow (2i,0)$. Now we look at $(2i-1,1)\in E_i$, an by the same lemma we obtain an edge $(2i-2,1)\rightarrow (2i-1,1)$. By an immediate induction we obtain right oriented edges from all points of $O_{i-1}$ to $E_i$ (and in particular this forces $\si_i=1$ as expected). 

\begin{figure}[!ht]
\includegraphics[width=0.8\textwidth]{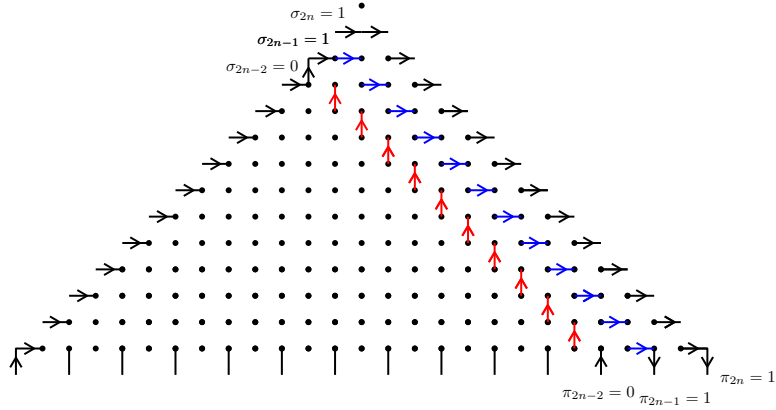}
\caption{Lemma~\ref{lem:commons} for $v=011$.
\label{fig:commonps2}}
\end{figure}

The lemma is thus proved by induction. \findem

There are other fixed edges than those stated in the lemma, which are pictured dashed on Figure~\ref{fig:commonsuf}, and which we explain now. For each $i=2n-1,\ldots,2n-|v|$, there are $2n-i$ extra fixed edges between points of $E_{i+k}$ and $O_{i+k}$, for $k=0,\ldots,2n-i-1$, as follows: If $i$ is such that $\si_i=\pi_i=0$, then these edges are at the top of the lines, and in particular the edge forced between $E_{2n-1}$ and $O_{2n-1}$ forces an extra $0$ at the beginning of $\tau$;  while if $\si_i=\pi_i=1$ the edges are at the bottom, and the fixed edge between $E_{2n-1}$ and $O_{2n-1}$ forces an extra $1$ at the end of $\tau$. 
Note that this shows in particular that $\tau$ has a prefix $0^{|v|_0+1}$ and a suffix $1^{|v_1|}$, which is the special case of Theorem~\ref{th:commonps} when $u=0$.

\begin{figure}[!ht]
\includegraphics[width=0.6\textwidth]{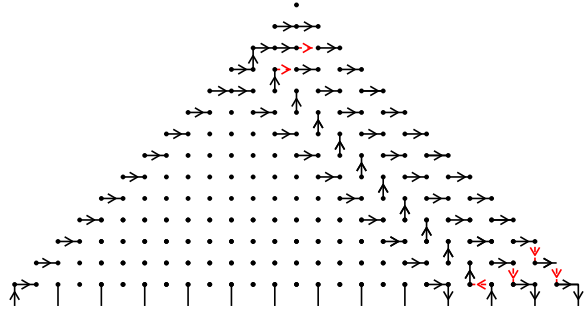}
\caption{Fixed edges for $v=011$.
\label{fig:commonsuf}}
\end{figure}

\medskip

\subsection{Common prefix} Now we consider the case of a common prefix (and we do not assume anything about a common prefix). We have the following lemma:

\begin{lemma}
\label{lem:commonp}
Let $u$ be a common prefix of $\si$ and $\pi$. Then
\begin{enumerate}
\item The edges situated southwest of $E_{|u|}$ are fixed by $\si$ and $\pi$. 
\item The edges situated strictly below the line $\{(x,y)~|~y=|u|_1\}$ and northeast of $E_{|u|}$) are fixed as zigzag paths.
\item There are fixed right oriented edges outgoing from the vertices $(|u|-1+i+j,|u|-1-i+j)$ for $0\leq i\leq |u|_0-1$ and $0\leq j\leq 2n-|u|$.
\end{enumerate}
\end{lemma}

\begin{figure}[!ht]
\includegraphics[width=0.6\textwidth]{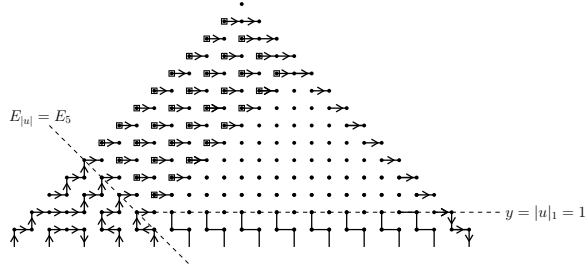}
\caption{Illustration of Lemma~\ref{lem:commonp} for $u=00100$.
\label{fig:commonpref}}
\end{figure}

\demof{Lemma~\ref{lem:commonp}} We advise the reader to look at the diagram \ref{fig:commonpref} to visualize the fixed regions indicated by the lemma. To prove the lemma we proceed by induction. It holds clearly for $u=0$, and assuming that it is true for a common prefix $u$ we now show that it holds also for $u0$ and $u1$.

{\bf Case $u0$}: Consider the vertex $(2|u|-1-|u|_1,|u|_1)\in O_{|u|}$; by the induction hypothesis, it has an incoming edge on its right, and because $\pi_{|u|+1}=0$ the fixed vertical edge under $(2|u|-|u|_1,|u|_1)\in O_{|u|}$ is incoming. By repeated applications of Lemma~\ref{lem:degierplus}, we get that:
\begin{itemize}
\item There are up edges from the points $(2|u|-1-|u|_1-k,|u|_1+k)\in O_{|u|},~k=0,\ldots,|u|_0-1$: this proves point $(1)$ of the lemma in this case;
\item There are right-oriented edges from the points $(2|u|-|u|_1+k,|u|_1+k)\in O_{|u|},~k=0,\ldots,2n-2-|u|$: this is point $(3)$ of the lemma.
\end{itemize}
 
There is nothing to prove for $(2)$, so we now consider the second case.

{\bf Case $u1$}:The starting vertex here is $(|u|,|u|-1)\in O_{|u|}$. Because $\si_{|u|+1}=0$, there is a fixed right oriented edge from this point by Lemma~\ref{lem:degierplus}, and repeated applications of the lemma entail other right oriented edges from $(|u|+k,|u|-1-k),k=1\ldots, 
|u|_0-1$, which proves point $(1)$ of the lemma. Moreover one has then fixed edges above and to the right of the points $(2|u|+1-|u|_1+k,|u|_1)$ for $k=0,\ldots,4n-2|u|-2$, which extends the zigzag paths and proves point $(2)$ Since point $(3)$ is immediate this time, the induction is complete and the lemma is proved. \findem

\begin{figure}[ht]
\begin{minipage}[b]{0.45\textwidth}
\centering
\includegraphics[width=\linewidth]{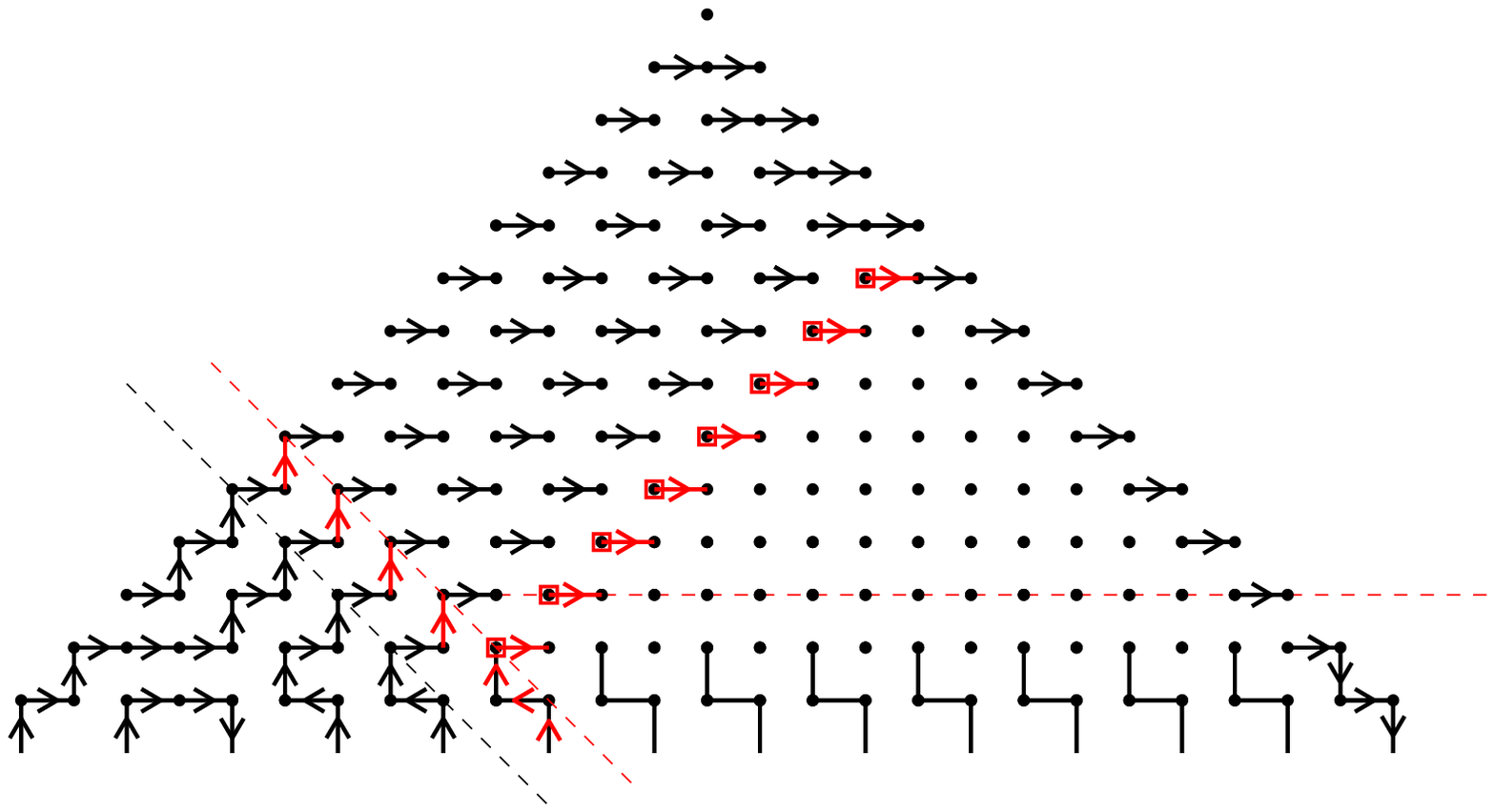}
\caption{$u0$
\label{fig:commonpref0}}
\end{minipage}
\begin{minipage}[b]{0.45\textwidth}
\centering
\includegraphics[width=\linewidth]{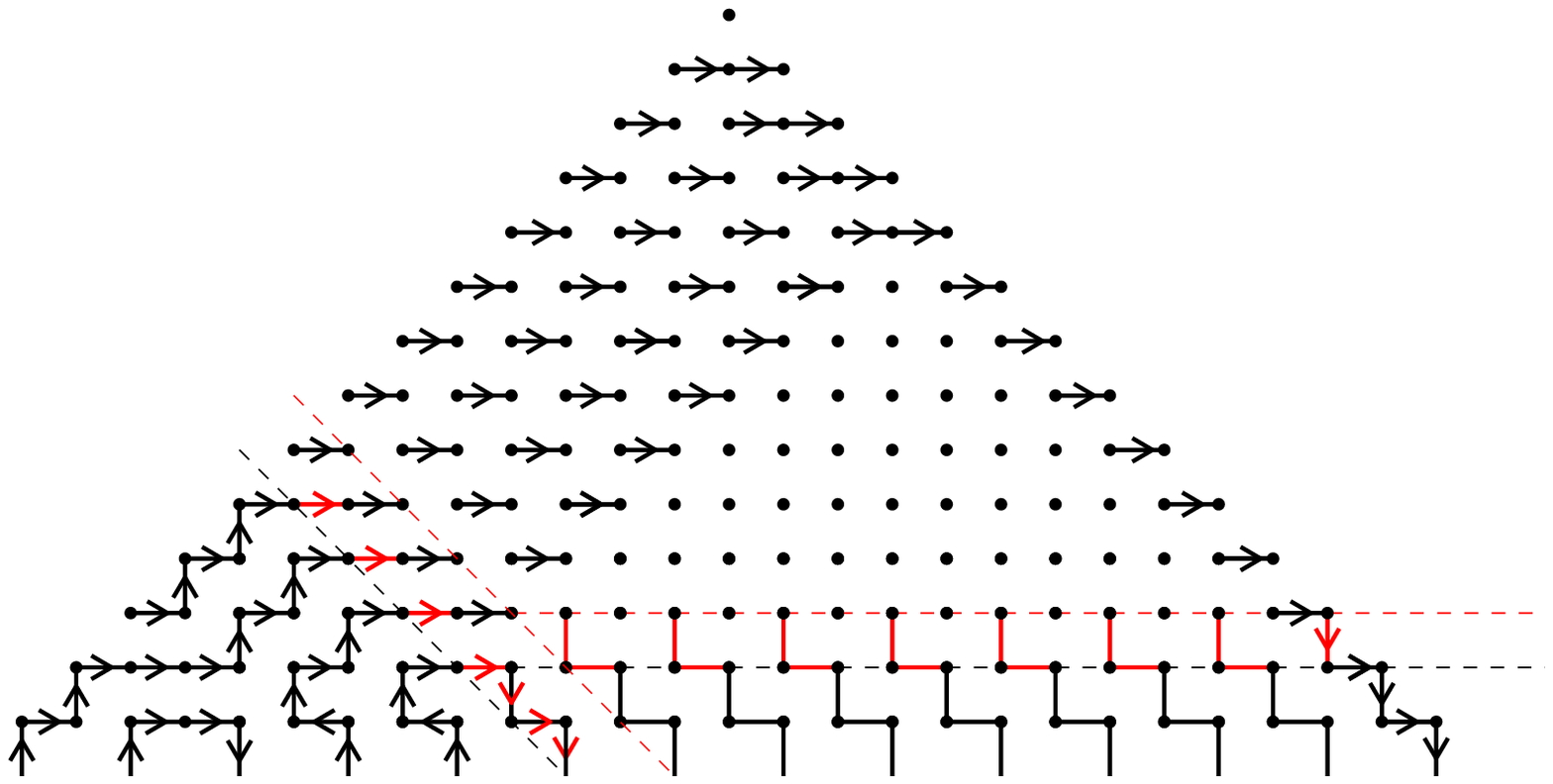}
\caption{$u1$
\label{fig:commonpref1}}
\end{minipage}
\end{figure}


\subsection{Common prefix and suffix}
In order to complete the proof of Theorem~\ref{th:commonps}, we cannot simply invoke the two sub-cases of common prefix and common suffix. We need to {\em merge} the two cases, and observe what happens when one does the union of all fixed edges. 

So we assume we have a common prefix $u$ and common suffix $v$. Lemma~\ref{lem:commons} give us fixed edges which we can draw drawn in the triangle; in particular the edges on the right boundary force $\tau$ to begin with $|v|_0+1$ zeros and end with $|v|_1$ ones. Lemma~\ref{lem:commonp}  also gives fixed edges, but we only consider those below the line $E_{2n-|v|}$, since this will simplify the proof. These fixed edges are represented by the black, non dotted edges in Figure~\ref{fig:commonps_final}.
\begin{figure}[!ht]
\includegraphics[width=0.7\textwidth]{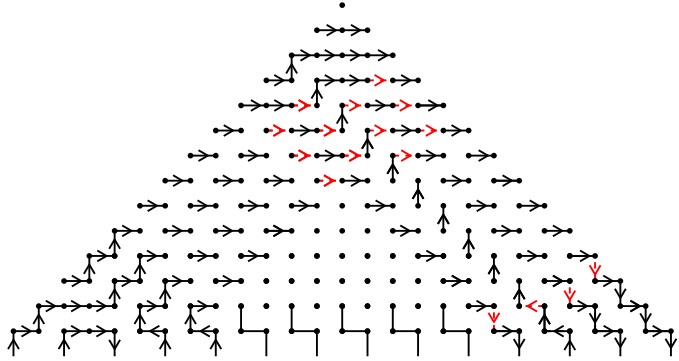}
\caption{Fixed edges for common prefix $u$ and common suffix $v$
\label{fig:commonps_final}}
\end{figure}

 Now consider what happens when the edges from points $(2)$ and $(3)$ of Lemma~\ref{lem:commonp} meet the line $E_{2n-|v|}$ and the fixed edges given by Lemma~\ref{lem:commons}: they force extra edges in the region, represented as red, dotted edges on Figure~\ref{fig:commonps_final}. More precisely, for $i$ between $2n-|v|$ and $2n-1$,
\begin{itemize}
\item  there are $|u|_0-1$ extra fixed oriented edges from $E_i$ to $O_i$ added in the Northwest of these lines; in particular, these force $|u|_0-1$ supplementary zeros in the prefix of $\tau$.                                                                                                                                                                                                                                                                                                                                                     \item  theres are  $|u|_1-1$ extra fixed oriented edges from $E_i$ to $O_i$ added in the Southeast of these lines; in particular, these force $|u|_1$ supplementary ones in the suffix of $\tau$.                                                                                                                                                                                                                                                                                                                                                                           \end{itemize}

In total we obtain that $\tau$ has necessarily a prefix composed of $|u|_0+|v|_0$ zeros and  $|u|_1+|v|_1$ ones, which achieves the proof of Theorem~\ref{th:commonps}.

\findem

\subsection{Special case of reverse shapes} We now go on to proving Theorem~\ref{th:rotatedpart}, naturally taking into account all fixed edges determined by the proof of Theorem~\ref{th:commonps}.

\demof{Theorem~\ref{th:rotatedpart}} We have still the factorizations $\si=u\si'v$ and $\pi=u\pi'v$, but in addition we suppose that $\pi'=1^a0^b$. This means that the external edges corresponding to $\pi'$ correspond to different loops, and here we appeal to the lemma of de Gier~\cite[Lemma 39]{degierloops} which gives new fixed edges\footnote{note that Lemma~\ref{lem:degierplus} is not sufficient here, because we do not have any information about the orientation of the edges.}.

\begin{figure}[!ht]
\includegraphics[width=0.6\textwidth]{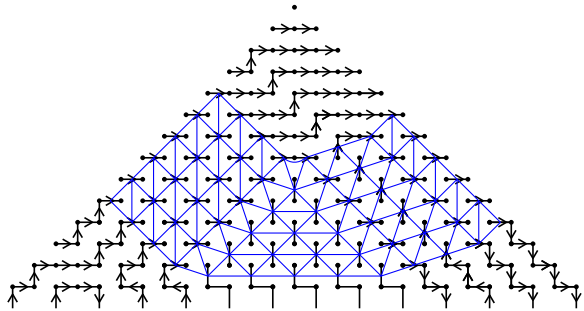}
\caption{Fixed edges in Theorem~\ref{th:rotatedpart} and the dual graph
\label{fig:rotatedpart_tiling}}
\end{figure}

Now one notices that all vertices of the triangle are adjacent to at least one fixed edge. In this case, we can also use an idea of de Gier~\cite{degierloops}: consider all vertices in the triangle which have exactly one incident edge; to form a TFPL configuration, one must find a matching between such vertices. Dually, we draw a polygon around each such vertex, and let two polygons be adjacent if the corresponding vertices can be matched. The result in our case is shown on Figure~\ref{fig:rotatedpart_tiling}, and note that here each polygon can be chosen to be a triangle, and that the result can be in fact embedded in the triangular lattice, cf. Figure~\ref{fig:rotatedtiling2}.
Now the number $\tspt$, which counts TFPL matchings, is equal to the number of rhombus tilings of the region with some dents induced by $\si$ and $\tau$, or more precisely by their factors $\si'$ and $\tau'$; indeed we just showed that each TFPL configuration will give rise to a certain tiling, and it is easily verified that each tiling of the region will correspond in return to a valid TFPL configuration. Rhombus tilings in such regions are equivalent to counting families of non intersecting lattice paths, the number of which can be written as a determinant thanks to the Lindstr\"om-Gessel-Viennot (LGV) method; see \cite{GV}. In our case, the possible starting points of the paths are $A_i=(i-1,i-1)$ and the possible ending points are $E_j=(2n-|u|_1-1,j-1)$, where $i,j=1\ldots a+b$.

Therefore, the number $\mathcal{P}(A_i\rightarrow E_j)$ of paths between $A_i$ and $E_j$ for given $i,j$ is
\[
  \mathcal{P}(A_i\rightarrow E_j)=\binom{2n-|u|_1-j}{i-j}
\]

Now the starting points $A_i$ are those for which $i\in I_1(\si')$, while the ending points are $E_{a+b+1-j}$ where $j\in I_1(\tau')$. The LGV lemma then tells us that $\tspt$ is equal to the determinant $\det\left(\mathcal{P}(A_i\rightarrow E_j)\right)$ where $i\in I_1(\si'), j \in  I_1(\tau')$, which is precisely the first determinant in the statement of Theorem~\ref{th:rotatedpart}.
For the second determinant, this corresponds to encoding the bijection between tilings and non intersecting lattice paths differently: instead of placing points on edges of the triangular lattice oriented Southwest to Northeast, we choose those that are oriented Southeast to Northwest. The rest of the demonstration then goes as before, and this achieves the proof of Theorem~\ref{th:rotatedpart}.\findem

\begin{figure}[!ht]
\includegraphics[width=\textwidth]{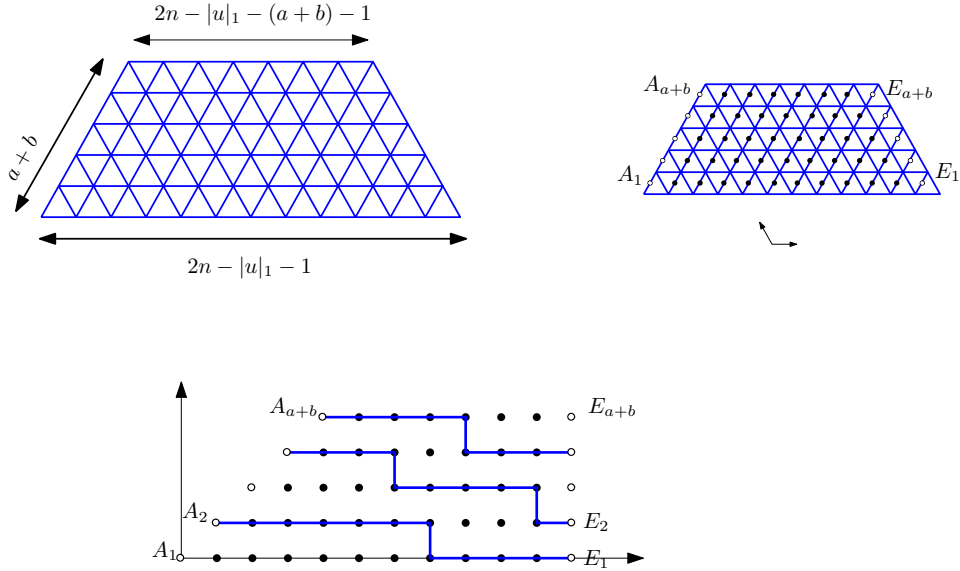}
\caption{Dual graph embedded in the triangular lattice.
\label{fig:rotatedtiling2}}
\end{figure}

\def\polhk\#1{\setbox0=\hbox{\#1}{{\o}oalign{\hidewidth
  \lower1.5ex\hbox{`}\hidewidth\crcr\unhbox0}}}

%

\begin{thebibliography}{10}

\bibitem{ProofRS}
Luigi Cantini and Andrea Sportiello.
\newblock Proof of the {R}azumov-{S}troganov conjecture.
\newblock {\em J. Combin. Theory Ser. A}, 118(5):1549--1574, 2011.

\bibitem{CasKrat}
F.~Caselli and C.~Krattenthaler.
\newblock Proof of two conjectures of {Z}uber on fully packed loop
  configurations.
\newblock {\em J. Combin. Theory Ser. A}, 108(1):123--146, 2004.

\bibitem{CKLN}
F.~Caselli, C.~Krattenthaler, B.~Lass, and P.~Nadeau.
\newblock On the number of fully packed loop configurations with a fixed
  associated matching.
\newblock {\em Electron. J. Combin.}, 11(2):Research paper 16, 2004.

\bibitem{degierloops}
J.~de~Gier.
\newblock Loops, matchings and alternating-sign matrices.
\newblock {\em Discrete Math.}, 298(1-3):365--388, 2005.

\bibitem{DZ4}
P.~{Di Francesco} and J.-B. Zuber.
\newblock On fully packed loop configurations with four sets of nested arches.
\newblock {\em J. Stat. Mech. Theory Exp.}, (6):005, 20 pp. (electronic), 2004.

\bibitem{artic47}
T.~Fonseca and P.~Zinn-Justin.
\newblock On some ground state components of the {O}(1) loop model.
\newblock {\em Journal of Statistical Mechanics: Theory and Experiment},
  2009(03):P03025 (29pp), 2009.

\bibitem{GV}
I.~Gessel and G.~Viennot.
\newblock Binomial determinants, paths, and hook length formulae.
\newblock {\em Adv. in Math.}, 58(3):300--321, 1985.

\bibitem{Kup-ASM}
G.~Kuperberg.
\newblock Another proof of the alternating-sign matrix conjecture.
\newblock {\em Internat. Math. Res. Notices}, (3):139--150, 1996.

\bibitem{MRR-Mac}
W.~Mills, D.~Robbins, and Jr.~H. Rumsey.
\newblock Proof of the {M}acdonald conjecture.
\newblock {\em Invent. Math.}, 66(1):73--87, 1982.

\bibitem{NadFPL2}
P.~Nadeau.
\newblock Fully {P}acked {L}oop configurations in a triangle {II}. {L}ittlewood
  {R}ichardson coefficients.
\newblock in preparation.

\bibitem{NadFPLFpsac}
P.~Nadeau.
\newblock {F}ully {P}acked {L}oop configurations in a triangle and
  {L}ittlewood--{R}ichardson coefficients.
\newblock {\em DMTCS Proceedings}, 0(01), 2010.

\bibitem{propp}
J.~Propp.
\newblock The many faces of alternating-sign matrices, 2001.

\bibitem{RS-conj}
A.~Razumov and Yu. Stroganov.
\newblock Combinatorial nature of the ground-state vector of the {$O(1)$} loop
  model.
\newblock {\em Teoret. Mat. Fiz.}, 138(3):395--400, 2004.

\bibitem{StanPP}
R.~P. Stanley.
\newblock Theory and application of plane partitions. {I}, {II}.
\newblock {\em Studies in Appl. Math.}, 50:167--188; ibid. 50 (1971), 259--279,
  1971.

\bibitem{Thapper}
J.~Thapper.
\newblock Refined counting of fully packed loop configurations.
\newblock {\em S{\'e}minaire Lotharingien de Combinatoire}, 56:B56e, 2007.

\bibitem{Wieland}
B.~Wieland.
\newblock A large dihedral symmetry of the set of alternating sign matrices.
\newblock {\em Electron. J. Combin.}, 7:Research Paper 37, 13 pp, 2000.

\bibitem{Zeil-ASM}
D.~Zeilberger.
\newblock Proof of the alternating sign matrix conjecture.
\newblock {\em Electron. J. Combin.}, 3(2):Research Paper 13, 84 pp, 1996.
\newblock The Foata Festschrift.

\bibitem{ZJtriangle}
P.~Zinn-Justin.
\newblock A conjectured formula for fully packed loop configurations in a
  triangle.
\newblock {\em Electron. J. Combin.}, 17(1):Research Paper 107, 2010.
\newblock arXiv:0911.4617v1.

\bibitem{Zuber-conj}
J.-B. Zuber.
\newblock On the {C}ounting of {F}ully {P}acked {L}oop {C}onfigurations: {S}ome
  new conjectures.
\newblock {\em Electron. J. Combin.}, 11(1):Research paper 13, 2004.

\end{thebibliography}

\end{document}